\documentclass{amsart}

\usepackage{amsfonts,amsthm,amsmath,tikz-cd,tikz,graphicx}

\usepackage[hidelinks]{hyperref}
\usepackage{xcolor}
\usepackage{libertine}

\newtheorem{theorem}{Theorem}[section]
\newtheorem{proposition}[theorem]{Proposition}
\newtheorem{lemma}[theorem]{Lemma}

\newtheorem{corollary}[theorem]{Corollary}

\theoremstyle{remark}
\newtheorem{remark}[theorem]{Remark}

\theoremstyle{definition}
\newtheorem{definition}[theorem]{Definition}

\newcommand{\Sq}{\mathcal{S}_q^{SL_3}}
\newcommand{\Sc}{\mathcal{S}_1^{SL_3}}
\newcommand{\hookuparrow}{\mathrel{\rotatebox[origin=c]{90}{$\hookrightarrow$}}}
\newcommand{\Oq}{\mathcal{O}_q(SL_3)}

\title{Miraculous cancellations and the quantum Frobenius for $SL_3$ skein modules}

\author[Higgins]{Vijay Higgins}
\address{Department of Mathematics, University of California, Los Angeles, CA 90095, USA}
\email{higginsv@math.ucla.edu}

\thanks{
	2020 {\em Mathematics Classification:} Primary 57K31. Secondary 17B37.}
	
\begin{document}
	\date{}

\begin{abstract}
We construct a quantum Frobenius map for the $SL_3$ skein module of any oriented 3-manifold specialized at a root of unity, and describe the map by way of threading certain polynomials along links. The homomorphism is a higher rank version of the Chebyshev-Frobenius homomorphism of Bonahon-Wong. The strategy builds on a previous construction of the Frobenius map for $SL_3$ skein algebras of punctured surfaces, using the Frobenius map of Parshall-Wang for the quantum group $\mathcal{O}_q(SL_3).$
\end{abstract}

\maketitle
\tableofcontents

\section{Introduction}

Skein modules of 3-manifolds were introduced independently by Turaev and Przytycki \cite{Tur88,Prz91} as a generalizations of the Jones polynomial of a link in $S^3.$ The most well-studied skein module is called the Kauffman bracket skein module and is constructed as the space spanned by isotopy classes of framed links in the 3-manifold modulo the Kauffman bracket skein relations. Another well-studied set of skein relations is the HOMFLY-PT skein relations, which define a link invariant which admits specializations to any of the $SL_n$ invariants of Reshetikhin-Turaev \cite{RT90}. The $SL_n$ skein module is a quantization of the $SL_n(\mathbb{C})$ character variety of the manifold \cite{Bul97,Sik05}.

Both the Kauffman bracket and HOMFLY-PT skein modules enjoy the property that the skeins and skein relations are naturally described in terms of framed links. The skein relations for $SL_n$ skein algebras  with $n\geq 3$ are most easily described in terms of certain graphs called webs \cite{Kup96,MOY98,Sik05,CKM14}. The Kauffman bracket skein relations correspond to the $SL_2$ case and are locally described by the classical Temperley-Lieb diagrams of planar matchings. In the case that the 3-manifold is a thickened oriented surface, it easily follows that the $SL_2$ skein module has a nice spanning set of nontrivial multicurves on the surface, which in fact turns out to be a basis \cite{Prz91,SW07}.

In the seminal works \cite{Kup94, Kup96}, Kuperberg used trivalent webs in a disk to construct the $SL_3$ analogue of Temperley-Lieb diagrams. Although the combinatorics of Kuperberg webs are richer and more complicated than the $SL_2$ case, Kuperberg showed that each of the defining relations fit into a linearly recursive procedure for simplifying webs. In this way the skein relations provide an algorithm to rewrite any web as a linear combination of canonically reduced webs which form a basis of so-called non-elliptic webs. Work of Sikora-Westbury \cite{SW07} showed that Kuperberg's rewriting procedure is confluent for webs on surfaces and consequently produces bases for the $SL_3$ skein modules of oriented surfaces. It remains an important problem to describe analogues of these web bases for skein modules of surfaces for the higher rank $n>3$ cases of $SL_n.$ 

When the manifold is a thickened oriented surface, the skein module of the manifold admits a natural algebra structure induced by superposition of skeins. For the case of $SL_2,$ the algebra is also known as the Kauffman bracket skein algebra and is well-studied at this point. Two of the most important breakthroughs in studying the $SL_2$ skein algebra were due to work of Bonahon-Wong. Their first breakthrough was the work \cite{BW11}, in which they constructed a so-called quantum trace map, embedding the skein algebra of any punctured surface into a quantum torus, which is the quantum Teichm\"uller space of Chekhov and Fock \cite{CF99}. In another work \cite{BW16}, they construct a so-called Chebyshev-Frobenius homomorphism which outputs central elements in the skein algebra at roots of unity. The Frobenius homomorphism has been used to great success in studying the representation theory of skein algebras at roots of unity \cite{BW16,FKBL19,GJS19}.

In \cite{Le18}, L\^{e} introduced a finer version of the skein algebra, called the stated skein algebra, which is built from stated tangles rather than links and satisfies a compatibility with the cutting and gluing of surfaces. Using the stated skein algebra, one can find easier reconstructions of the quantum trace map \cite{Le18, CL19} and of the Chebyshev-Frobenius homomorphism \cite{KQ19, BL22}. The stated skein algebra also illuminates a connection between the skein algebra and the quantum group $\mathcal{O}_q(SL_2).$

In \cite{Hig23} a stated skein algebra was constructed for $SL_3$ which extended the skein algebra of Kuperberg webs to one that is compatible with the cutting of punctured surfaces. One important aspect of the construction is that it showed that the Sikora-Westbury basis \cite{SW07} of Kuperberg's non-elliptic webs for the ordinary $SL_3$ skein algebra can be extended to a basis for the stated skein algebra. Independently, \cite{FS22} and \cite{DS24} constructed coordinates for Kuperberg basis webs for punctured surfaces. The work of Kim \cite{Kim22} used these ingredients to construct an $SL_3$ analogue of Bonahon-Wong's quantum trace map, and showed that it was an embedding. Towards a study of higher rank $SL_n$ skein algebras, work of L\^{e}-Sikora \cite{LS22} constructed an $SL_n$ stated skein algebra and splitting homomorphisms which the work of L\^{e}-Yu \cite{LY23} used to construct an $SL_n$ analogue of the quantum trace map. One key ingredient that is currently missing from the study of $SL_n$ skein algebras for $n>3$ is the construction of bases as well as a proof that the splitting homomorphisms and quantum trace maps are injective in the case of punctured surfaces.

In \cite{BH23} a threading operation was defined for links in the $SL_n$ skein algebra which produces central elements in $\mathcal{S}_q^{SL_n}(\Sigma)$ when $q$ is a suitable root of unity. In \cite{Hig23} an algebra homomorphism $F_\Sigma : \Sc(\Sigma) \rightarrow \Sq(\Sigma)$ was constructed when $q$ is a suitable root of unity and every connected component of $\Sigma$ has at least one puncture. That construction used the Frobenius homomorphism $\mathcal{O}_1(SL_3) \rightarrow \mathcal{O}_q(SL_3)$ of Parshall and Wang \cite{PW91}, but did not give a manageable way to compute the image of the map on a link. The goal of this paper is to show that for $SL_3$ the Frobenius homomorphism of \cite{Hig23} applied to a link coincides with the threading operation of \cite{BH23}. Consequently, we show that the threading operation respects the skein relations for $SL_3$ and defines a homomorphism of skein modules.

\subsection{Main results}

The construction of the Chebyshev-Frobenius homomorphism for the $SL_2$ or Kauffman bracket skein algebra \cite{BW16} relied on so-called `miraculous cancellations'  which were incarnations of an identity holding in the quantum group $\mathcal{O}_q(SL_2)$ at a specialization of the quantum parameter $q$ to a root of unity. Suppose that $T_N$ is the $N\text{-th}$ Chebyshev polynomial of the first kind, which is defined by the property that for any matrix $A$ in the (classical) algebraic group $SL_2$ we have

\begin{equation*}
T_N(\text{tr}(A))=\text{tr}(A^N).
\end{equation*}

Now let $a$ and $d$ denote the diagonal elements of the standard quantum matrix of generators $\begin{pmatrix}
	a&b\\
	c&d
\end{pmatrix}$ for the quantum group $\mathcal{O}_q(SL_2).$ When $q^{1/2}$ is a root of unity of odd order $N,$ the following identity holds in $\mathcal{O}_q(SL_2):$

\begin{equation*}
T_N(a+d)=a^N+d^N.
\end{equation*}

The identity of this particular form appears in \cite{KQ19,BL22} but see also \cite{BW16,Bon19,LP19, Le15, BR22} for variations and applications.  

Later in Definition \ref{power sum}, for $N\geq 1,$ we describe a polynomial $P^{(N)}(x,y) \in \mathbb{Z}[x,y]$ which is the $SL_3$ analogue of $T_N.$ For $SL_3$ it is also called the power sum polynomial and is the $n=3$ specialization of the polynomials used in \cite{BH23} to construct central elements in $SL_n$ skein algebras. We prove the polynomial $P^{(N)}$ satisfies the following identities.

\begin{theorem}\label{cancellations}(Corollary \ref{miraculous cancellations})
Let $X$ be the standard quantum matrix of generators $X_{ij}$ for the quantum group $\mathcal{O}_q(SL_3)$ defined over a commutative ring $\mathcal{R}$ containing an invertible element $q^{1/6}.$ Let $\sigma_1$ denote the expression obtained by taking the trace of $X$ and let $\sigma_2$ denote the expression obtained by taking the sum of the $2\text{-by-}2$ principal quantum minors of $X.$

When $q^{1/3}$ is a root of unity of order $N$ coprime to $6$ we have the following identities in $\mathcal{O}_q(SL_3):$
\begin{align*}
P^{(N)}(\sigma_1,\sigma_2)=&X_{11}^N+X_{22}^N+X_{33}^N\\
P^{(N)}(\sigma_2,\sigma_1)=&X_{11}^NX_{22}^N+X_{11}^NX_{33}^N+X_{22}^NX_{33}^N\\
&- X_{12}^NX_{21}^N-X_{13}^NX_{31}^N-X_{23}^NX_{32}^N.
\end{align*}

\end{theorem}

Using Theorem \ref{cancellations} along with certain fundamental constructions from the $SL_3$ stated skein algebras from \cite{Hig23}, we are able to compute in the case of the annulus $\mathcal{A},$ the image of the Frobenius homomorphism $F_{\mathcal{A}}$ on an essential loop on the annulus. Using a compatibility of the embedding of an annulus in a thickened surface with the cutting of that surface (Lemma \ref{cutting and gluing annulus}), we are able to build up to the following main result.

\begin{theorem}\label{intro main thm}(Theorem \ref{main theorem})
Suppose that $M$ is an oriented 3-manifold. When $q^{1/3} \in \mathcal{R}$ is a root of unity of order $N$ coprime to $6,$ then the threading operation
\begin{equation*}
	L \mapsto L^{[P^{(N)}]}
\end{equation*}
that sends a framed oriented link $L$ in $M$ to the result of threading $P^{(N)}$ along each component of the link defines a homomorphism of skein modules $\mathcal{S}_1^{SL_3}(M) \rightarrow \mathcal{S}_q^{SL_3}(M).$ Furthermore, in the particular case when $M=\Sigma \times (-1,1)$ is a thickened oriented surface, then the threading operation defines an algebra homomorphism of skein algebras $\mathcal{S}_1^{SL_3}(\Sigma) \rightarrow Z(\mathcal{S}_q^{SL_3}(\Sigma)).$ In case the surface has at least one puncture on each connected component, the resulting algebra map coincides with the Frobenius map defined in \cite{Hig23} and is thus known to be injective.
\end{theorem}

For alternate approaches to the construction of Frobenius homomorphisms for various versions of skein modules, see \cite{GJS19,BR22,KLW24,CKL24}. For the case of $SL_n$ stated skein algebras of surfaces with boundary, Wang has described the Frobenius homomorphism in terms of framed powers of stated arcs in \cite{Wang23}. During preparation of this paper, it was announced in \cite{KW24} that Theorem \ref{intro main thm} will also appear in an upcoming work \cite{KLW24}.

The family of power sum polynomials $P^{(N)}$ used in this paper is an $SL_3$ analogue of the Chebyshev polynomials of the first kind. The Chebyshev polynomials of the first and second kind are ubiquitious in skein theory and have famously appeared in studying the skein algebra of the annulus \cite{Lick97} and the torus \cite{FG00}. They also appear in relation to cluster algebras, positive bases, and categorification \cite{Thur14,Queff22,MQ23}. Higher rank analogues of these polynomials, often referred to as power sum polynomials, have also been used in HOMFLY-PT and Kauffman skein theory \cite{MS17,MPS23}. In $G_2$ skein theory, analogues of power sum polynomials have been used to construct central elements in skein algebras \cite{BBHHMP23}.

We expect that the map from Theorem \ref{intro main thm} will have applications to the representation theory of $SL_3$ skein algebras. For instance, in \cite{BH23} it was conjectured that for a closed surface $\Sigma,$ the image of the Frobenius map $F_\Sigma$ is equal to the center of the skein algebra $Z(\Sq(\Sigma)).$ Indeed, the recent work of \cite{KW24} contains work in this direction. The elements obtained by applying $F_\Sigma$ to basis webs might also have applications to the categorification of $SL_3$ skein algebras.

\subsection{Acknowledgments}
This work was completed while in the Mathematics Department at Michigan State University. I would like to thank Francis Bonahon for many helpful conversations. I am grateful for travel support provided by the Topology group at Michigan State University from the NSF RTG grant DMS-2135960, which allowed me several opportunities to speak about this work at conferences.

\section{The quantum group $\mathcal{O}_q(SL_3)$}

Let $\mathcal{R}$ be a commutative ring containing a distinguished invertible element $q^{1/6}.$ The element $q$ is called the quantum parameter and we are most interested in the case when $q$ is a root of unity. Whenever we speak of the specialization $q=1,$ we will use the convention that $q^{1/6}$ is chosen to be $1$ as well.

We begin by recalling definition of the quantum group $\mathcal{O}_q(SL_3)$ and then the Frobenius map described by Parshall-Wang for the quantum group $\mathcal{O}_q(SL_3)$ when $q$ is a root of unity.

\begin{definition}\label{Oq}
The quantum group $\mathcal{O}_q(SL_3)$ is the quotient of the algebra freely generated over $\mathcal{R}$ by the variables $\{X_{ij} \mid 1 \leq i,j \leq 3\}$ modulo the quantum matrix relations

\begin{equation*}
	X_{ij}X_{lm} = \begin{cases}
		q^{-1}X_{lm}X_{ij} & (i>l, j=m),\\
		q^{-1}X_{lm}X_{ij} & (i=l, j>m),\\
		X_{lm}X_{ij} & (i<l, j>m),\\
		X_{lm}X_{ij}-(q-q^{-1})X_{im}X_{lj},
	\end{cases}
\end{equation*}

along with the quantum determinant relation

\begin{equation*}
	1=\sum_{\sigma \in S_3} (-q)^{l(\sigma)} X_{\sigma(1)1}X_{\sigma(2)2}X_{\sigma(3)3},\\
\end{equation*} where the sum ranges over permutations $\sigma$ and $l(\sigma)$ is the length, or number of inversions, of the permutation $\sigma.$
\end{definition}

The quantum group $\mathcal{O}_q(SL_3)$ is defined for any invertible choice of quantum parameter $q \in \mathcal{R}.$ When $q=1,$ the algebra $\mathcal{O}_1(SL_3)$ is commutative and coincides with the coordinate ring of the matrix group $SL_3.$ The Hopf algebra structure maps for $\mathcal{O}_q(SL_3)$ are inspired by matrix operations. The coproduct $\Delta: \mathcal{O}_q(SL_3) \rightarrow \mathcal{O}_q(SL_3) \otimes \mathcal{O}_q(SL_3)$, counit $\varepsilon: \mathcal{O}_q(SL_3)\rightarrow \mathcal{R}$ and antipode $S:\mathcal{O}_q(SL_3) \rightarrow \mathcal{O}_q(SL_3)$ are given by

\begin{align*}
	\Delta(X_{ij})&=\sum_{r=1}^3 X_{ir}\otimes X_{rj},\\
	\varepsilon(X_{ij})&=\delta_{ij},\\
	S(X_{ij})&=(-q)^{i-j}A[j\vert i],
\end{align*}

where $A[j \vert i]$ denotes the quantum minor of the matrix of generators $X=(X_{ij})$ after deleting row $j$ and column $i.$ See the textbook of Brown-Goodearl \cite{BG02}, for example, for more details about quantum groups $\mathcal{O}_q(G),$ which are dual to the quantum groups $U_q(\mathfrak{g}).$

We will be making use of two special elements of $\Oq,$ which we will call $\sigma_1$ and $\sigma_2.$ We obtain $\sigma_1$ by taking the trace of the matrix of generators, and we obtain $\sigma_2$ by taking the sum of the $2\text{-by-}2$ principal quantum minors of the matrix of generators:

\begin{align*}
\sigma_1:=&X_{11}+X_{22}+X_{33}\\
\sigma_2:=&X_{11}X_{22}-qX_{12}X_{21}\\
&+X_{22}X_{33}-qX_{23}X_{32}\\
&+X_{11}X_{33}-qX_{13}X_{31}.
\end{align*}

When $q=1$ and entries of $X$ are replaced by elements in an algebraically closed field, then $\sigma_1$ and $\sigma_2$ provide formulas for the first and second elementary symmetric sums of the eigenvalues of $X.$ Although quantum matrices do not have eigenvalues in the classical sense, the expressions $\sigma_1$ and $\sigma_2$ often are referred to as quantum spectral data.

\subsection{The Frobenius map for $\mathcal{O}_q(SL_3)$}

Parshall and Wang described the following Hopf algebra map $\mathcal{O}_1(SL_3) \rightarrow \mathcal{O}_q(SL_3),$ called the Frobenius homomorphism, when $q$ is a root of unity in \cite{PW91}. This map is dual to the Frobenius maps for universal enveloping algebras $U_q(sl_3) \rightarrow U_1(sl_3)$ described by Lusztig \cite{Lus90}.

\begin{theorem}\label{Frobenius Oq}(\cite{PW91})
	If $q$ and $q^2$ are both roots of unity of order $N,$ then there exists a Hopf algebra embedding $\mathcal{O}_1(SL_3)$ into the center of $\mathcal{O}_q(SL_3),$ defined on the standard generators $X_{ij}$  by taking $N\text{-th}$ powers:
	
	\begin{align*}
		F_{\mathcal{O}_q}: \mathcal{O}_1(SL_3) &\hookrightarrow Z(\mathcal{O}_q(SL_3))\\
		X_{ij} &\mapsto (X_{ij})^N.
	\end{align*}
\end{theorem}

We will call the algebra homomorphism $F_{\mathcal{O}_q}$ the Frobenius map for $\Oq.$ In the following section we describe the skein theory for $SL_3$ and describe the connection between skein algebras of surfaces and $\mathcal{O}_q(SL_3)$ which allows us to build a Frobenius homomorphism for skein modules from the Frobenius homomorphism for $\mathcal{O}_q(SL_3).$

\section{$SL_3$ skein modules and algebras}

\subsection{$SL_3$ webs}

An abstract $SL_3$ web $W$ is an oriented graph such that each vertex is a trivalent source or a trivalent sink. It is permitted for $W$ to have connected components which are vertexless oriented loops. It is also permitted for $W$ to be empty.

An $SL_3$ web $W$ in an oriented 3-manifold $M$ is an embedding of $W$ in $M$ equipped with an oriented surface $\Sigma_W$ embedded in $M$ so that the surface contains $W$ and deformation retracts onto $W.$ We further require each edge of $W$ to carry a vertical framing relative to $\Sigma_W.$ If the underlying graph of $W$ has no vertex, then $W \subset M$ is the same as a framed link in $M.$

\begin{definition}\label{skein module}
For an oriented 3-manifold $M,$ the $SL_3$ \textit{skein module of} $M$ is denoted by $\Sq(M)$ and is the quotient of the free $\mathcal{R}\text{-module}$ spanned by isotopy classes of webs $W$ in $M$ subject to the following skein relations of Kuperberg \cite{Kup94,Kup96}.

\begin{align}
	\begin{tikzpicture}[baseline=3ex]
		\node (center) at (1/4,1/2) {};
		\draw[->] (1/2,1) -- (0,0);
		\draw (0,1) -- (center);
		\draw [->] (center)--(1/2,0);
	\end{tikzpicture}&=
	q^{2/3}\begin{tikzpicture}[baseline=3ex]
		\draw [->] (0,1)--(0,0);
		\draw [->] (1/2,1)--(1/2,0);
	\end{tikzpicture} +q^{-1/3} \begin{tikzpicture}[baseline=3ex]
		\draw [<-] (0,0)--(1/4,1/4);
		\draw [<-] (1/2,0)--(1/4,1/4);
		\draw (1/4,1/4)--(1/4,3/4);
		\draw [-<] (1/4,3/4)--(1/4,1/2);
		\draw (1/4,3/4)--(0,1);
		\draw (1/4,3/4)--(1/2,1);
		\draw [->] (0,1)--(1/8,7/8);
		\draw [->] (1/2,1)--(3/8,7/8);
	\end{tikzpicture}&
	\begin{tikzpicture}[baseline=3ex]
		\draw [-<] (1/4,0)--(1/4,1/8);
		\draw (1/4,0)--(1/4,1/4);
		\draw [-<] (1/4,3/4)--(1/4,7/8);
		\draw (1/4,3/4)--(1/4,1);
		\draw [->](1/4,1/4) to [out=135, in=-90] (0,1/2);
		\draw (0,1/2) to [out=90, in =-135] (1/4,3/4);
		\draw [->] (1/4,1/4) to [out=45, in=-90] (1/2,1/2);
		\draw (1/2,1/2) to [out=90, in=-45] (1/4,3/4);
	\end{tikzpicture}&=-(q+q^{-1}) \hspace{.1cm} \begin{tikzpicture}[baseline=3ex]
		\draw [->] (0,1)--(0,1/2);
		\draw (0,1)--(0,0);
	\end{tikzpicture} \label{2gon}\\
	\begin{tikzpicture}[baseline=3ex]
		\node (center) at (1/4,1/2) {};
		\draw[->] (0,1) -- (1/2,0);
		\draw (1/2,1) -- (center);
		\draw [->] (center)--(0,0);
	\end{tikzpicture}&=
	q^{-2/3}\begin{tikzpicture}[baseline=3ex]
		\draw [->] (0,1)--(0,0);
		\draw [->] (1/2,1)--(1/2,0);
	\end{tikzpicture} +q^{1/3} \begin{tikzpicture}[baseline=3ex]
		\draw [<-] (0,0)--(1/4,1/4);
		\draw [<-] (1/2,0)--(1/4,1/4);
		\draw (1/4,1/4)--(1/4,3/4);
		\draw [-<] (1/4,3/4)--(1/4,1/2);
		\draw (1/4,3/4)--(0,1);
		\draw (1/4,3/4)--(1/2,1);
		\draw [->] (0,1)--(1/8,7/8);
		\draw [->] (1/2,1)--(3/8,7/8);
	\end{tikzpicture}&
	\begin{tikzpicture}[baseline=-4ex]
		\draw [>-] (0,0) arc [radius=.4, start angle=90, end angle=465];
	\end{tikzpicture}&=q^2+1+q^{-2} \label{crossing}\\
	\begin{tikzpicture}[baseline=3ex]
		\draw [->] (0,0)--(0,1/8);
		\draw [-<] (0,1/8)--(0,1/2);
		\draw [->] (0,0)--(0,7/8);
		\draw (0,7/8)--(0,1);
		\draw [-<] (1/2,0)--(1/2,1/8);
		\draw [->] (1/2,0)--(1/2,1/2);
		\draw [-<] (1/2,1/2)--(1/2,7/8);
		\draw (1/2,0)--(1/2,1);
		\draw [-<] (0,1/4)--(1/4,1/4);
		\draw (0,1/4)--(1/2,1/4);
		\draw [->] (0,3/4)--(1/4,3/4);
		\draw (1/4,3/4)--(1/2,3/4);
	\end{tikzpicture}&=\begin{tikzpicture}[baseline=3ex]
		\draw [->] (0,0)--(0,1/2);
		\draw (0,0)--(0,1);
		\draw [-<] (1/2,0)--(1/2,1/2);
		\draw (1/2,0)--(1/2,1);
	\end{tikzpicture} + \begin{tikzpicture}[baseline=3ex]
		\draw [<-] plot [smooth] coordinates {(0,1)(1/4,3/4)(1/2,1)};
		\draw [->] plot [smooth] coordinates {(0,0)(1/4,1/4)(1/2,0)};
	\end{tikzpicture}&\begin{tikzpicture}[baseline=-4ex]
		\draw [<-] (0,0) arc [radius=.4, start angle=90, end angle=465];
	\end{tikzpicture}&=q^2+1+q^{-2} \label{loop}
\end{align}	
\end{definition}

Each skein relation takes place in a small ball and is applicable to webs which are equal outside of the ball. The skein relations are depicted with vertical framing. Let $I=(-1,1)$ denote the interval. In the particular case that $M=\Sigma \times I$ is the thickening of an oriented surface $\Sigma,$ the skein module $\Sq(\Sigma \times I)$ will be denoted by the shorter notation $\Sq(\Sigma).$ The module $\Sq(\Sigma)$ has a natural algebra structure, given by superposition, such that product of two elements $[W_1],[W_2] \in \Sq(\Sigma)$ represented by webs $W_1,W_2$ is given by

\begin{equation*}
[W_1][W_2]=[W_1' \cup W_2'],
\end{equation*}
 
where $W_1'$ is obtained by isotoping $W_1$ so that is contained in $\Sigma \times (0,1),$ and $W_2'$ is obtained by isotoping $W_2$ so that it is contained in $\Sigma \times (-1,0).$ When an element in $\Sq(M)$ is represented by a web $W$ in $M,$ we will often denote the element $[W]\in \Sq(M)$ simply by $W$ when it does not cause confusion. The product is generally a noncommutative product, but there are special cases when $\Sq(\Sigma)$ is commutative. One important example is the case $q=1,$ in which case $\Sc(\Sigma)$ is commutative for any surface $\Sigma.$ Another case is when $\Sigma=\mathcal{A}$ is the annulus, in which case $\Sq(\mathcal{A})$ is commutative for any $q.$

For certain simple examples such as $M=S^3$ or $M=S^2 \times I,$ the skein module is isomorphic to the ground ring $\mathcal{R},$ and any link in the manifold can be evaluated to an element of $\mathcal{R},$ which is its $SL_3$ link invariant, a specialization of the HOMFLY-PT polynomial of the link.

\subsection{Webs and the $SL_3$ representation category} Kuperberg first introduced the relations (\ref{2gon})-(\ref{loop}) in \cite{Kup94} in a study of the linearly recurrent nature of link invariants for the rank 2 Lie groups. In a second paper, Kuperberg showed in \cite{Kup96} that the spider, or web category, of $SL_3$ webs describes all morphisms of $U_q(sl_3)\text{-modules}$ (equivalently, $\mathcal{O}_q(SL_3)\text{-comodules}$) between tensor products of the 3-dimensional vector representation and its dual. He further showed that the relations (\ref{2gon})-(\ref{loop}) generate all relations among the morphisms. Kuperberg's result was proven for $\mathcal{R}=\mathbb{C}$ with $q$ either a generic parameter or specialized to $q=1,$ in which cases the category $U_q(sl_3)\text{-mod}$ is semisimple and can be obtained from the web category by a process called idempotent completion. Analogues of Kuperberg's result still hold when $\mathcal{R}$ is any commutative ring and for any invertible $q$ \cite{Hig23,Eli15}.

\subsection{Functoriality of skein modules}\label{Functoriality skein module}

A simple but important property for skein modules is the following functoriality. Suppose that $M,N$ are oriented 3-manifolds and that $i: M\hookrightarrow N$ is an orientation preserving embedding of $M$ into $N.$ The embedding induces a homomorphism of skein modules $i_*: \Sq(M) \rightarrow \Sq(N)$ defined on webs $W$ in the obvious way by setting $i_*([W])=[i(W)].$ The map $i_*$ is well-defined since it takes skein relations in $M$ to skein relations in $N.$ Although $i$ is an embedding, the induced homomorphism $i_*$ need not be injective.

Another apparent but important fact is that if $M_1$ and $M_2$ are disjoint oriented 3-manifolds, then $\Sq(M_1 \sqcup M_2) \cong \Sq(M_1) \otimes \Sq(M_2)$ as $\mathcal{R}\text{-modules}.$ Similarly, if $\Sigma_1$ and $\Sigma_2$ are disjoint oriented surfaces, then $\Sq(\Sigma_1 \sqcup \Sigma_2) \cong \Sq(\Sigma_1) \otimes \Sq(\Sigma_2)$ as $\mathcal{R}\text{-algebras}.$ Together, these facts give that the skein module gives a monoidal functor $\Sq(-)$ on the category whose objects are oriented 3-manifolds and whose morphisms are isotopy classes of orientation-preserving embeddings. 

The main results of this paper concern the skein modules of oriented 3-manifolds and the skein algebras of thickened surfaces. We will use the functoriality of skein modules to use embedded (thickened) surfaces to study skein modules of 3-manifolds. A tool which will allow us to use quantum groups to study the skein algebras of thickened surfaces is the stated skein algebra of a surface with boundary.

\subsection{The $SL_3$ stated skein algebra}

The $SL_3$ stated skein algebra was introduced in \cite{Hig23}, generalizing the Kauffman bracket stated skein algebra introduced by L\^{e} in \cite{Le18}. Here, we will extend the definition of $\Sq(\Sigma)$ from Definition \ref{skein module} to surfaces with boundary.

Let $\Sigma=\Sigma' \setminus \mathcal{P}$ be obtained from a compact oriented surface $\Sigma',$ possibly with boundary, by removing a finite set of points $\mathcal{P}$ from $\Sigma'$ so that for each boundary circle $C_i \subset \partial \Sigma',$ we have $C_i \cap \mathcal{P} \neq \emptyset.$ This guarantees that each boundary component of $\Sigma$ is an interval. Such a surface $\Sigma$ in our context is referred to as a punctured bordered surface \cite{Le18}.

We now consider webs which may have univalent endpoints contained in $\partial \Sigma \times I$ such that each endpoint on the same connected component of $\partial \Sigma \times I$ occur at distinct heights. We require the edge of the web which is adjacent to the endpoint to meet the boundary transversely and with vertical framing. We consider such webs up to isotopy preserving the height order on each boundary component. A \textit{state} for a web $W$ is a function $s: \partial W \rightarrow \{1,2,3\}$. A \textit{stated web} $(W,s)$ is a web $W$ equipped with a state $s.$ We can depict stated webs by taking a generic projection of the stated web $W \subset \Sigma \times I$ onto the surface $\Sigma$ and recording the height order of the endpoints by drawing an arrow along the boundary interval to indicate that endpoints further along the interval occur at greater heights in $\partial \Sigma \times I.$ 

\begin{definition}
The $SL_3$ \textit{stated skein algebra} of $\Sigma,$ also denoted by $\Sq(\Sigma)$,  is the quotient of the $\mathcal{R}\text{-module}$ freely spanned by isotopy classes of stated webs subject to the Kuperberg relations (\ref{2gon})-(\ref{loop}) in addition to the following relations (\ref{switch})-(\ref{spike}) along the boundary. The product for the stated skein algebra is again defined by superposition of stated webs. In the case that $\Sigma$ has no boundary, this definition of $\Sq(\Sigma)$ coincides with Definition \ref{skein module}.
\end{definition}

Along each boundary interval we impose the following stated skein relations for any $a,b \in \{1,2,3\}$ such that $a>b.$

\begin{align}
\begin{tikzpicture}[baseline=.5ex]
	\draw [line width =1.5, ->] (-1/4,0)--(3/4,0);
	\draw (0,1)--(0,0);
	\draw (1/2,1)--(1/2,0);
	\draw [->] (0,1)--(0,1/2);
	\draw [->] (1/2,1)--(1/2,1/2);
	\node [below] at (0,0) {$b$};
	\node [below] at (1/2,-.07) {$a$};
\end{tikzpicture}&=q^{-1} \begin{tikzpicture}[baseline=.5ex]
	\draw [line width =1.5, ->] (-1/4,0)--(3/4,0);
	\draw (0,1)--(0,0);
	\draw (1/2,1)--(1/2,0);
	\draw [->] (0,1)--(0,1/2);
	\draw [->] (1/2,1)--(1/2,1/2);
	\node [below] at (0,-.07) {$a$};
	\node [below] at (1/2,0) {$b$};
\end{tikzpicture}+\begin{tikzpicture}[baseline=.5ex]
	\draw (0,0)--(1/4,1/4);
	\draw [->] (1/4,1/4)--(3/8,1/8);
	\draw (1/2,0)--(1/4,1/4);
	\draw [->] (1/4,1/4)--(1/8,1/8);
	\draw (1/4,1/4)--(1/4,3/4);
	\draw [-<] (1/4,3/4)--(1/4,1/2);
	\draw (1/4,3/4)--(0,1);
	\draw (1/4,3/4)--(1/2,1);
	\draw [->] (0,1)--(1/8,7/8);
	\draw [->] (1/2,1)--(3/8,7/8);
	\draw [line width=1.5, ->] (-1/4,0)--(3/4,0);
	\node [below] at (0,-.07) {$a$};
	\node [below] at (1/2,0) {$b$};
\end{tikzpicture}&\begin{tikzpicture}[baseline=.5ex]
\draw [line width=1.5, ->] (-1/4,0)--(3/4,0);
\draw (0,0)--(1/4,1/2);
\draw (1/2,0)--(1/4,1/2);
\draw (1/4,1/2)--(1/4,1);
\draw [-<] (0,0)--(1/8,1/4);
\draw [-<] (1/2,0)--(3/8,1/4);
\draw [-<] (1/4,1)--(1/4,3/4);
\node [below] at (0,0) {$a$};
\node [below] at (1/2,0) {$a$};
\end{tikzpicture}&=0 \label{switch}\\
\begin{tikzpicture}[baseline=.5ex]
	\draw [line width=1.5, ->] (-1/4,0)--(3/4,0);
	\draw (0,0)--(1/4,1/2);
	\draw (1/2,0)--(1/4,1/2);
	\draw (1/4,1/2)--(1/4,1);
	\draw [-<] (0,0)--(1/8,1/4);
	\draw [-<] (1/2,0)--(3/8,1/4);
	\draw [-<] (1/4,1)--(1/4,3/4);
	\node [below] at (0,-.07) {$a$};
	\node [below] at (1/2,0) {$b$};
\end{tikzpicture}&=q^{-7/6}\begin{tikzpicture}[baseline=.5ex]
	\draw [line width=1.5, ->] (-1/4,0)--(3/4,0);
	\draw [->] (1/4,0)--(1/4,1/2);
	\draw (1/4,0)--(1/4,1);
	\node [below] at (1/4,0) {$a+b-2$};
\end{tikzpicture} & \begin{tikzpicture}[baseline=.5ex]
\draw (0,1/2)--(0,0);
\draw [->] (0,1/2)--(0,1/4);
\draw (1/4,0)--(0,1/2);
\draw [->] (0,1/2)--(1/8,1/4);
\draw (-1/4,0)--(0,1/2);
\draw [->] (0,1/2)--(-1/8,1/4);
\draw [line width=1.5,->] (-1/2,0)--(1/2,0);
\node [below] at (-1/4,0) {$3$};
\node [below] at (0,0) {$2$};
\node [below] at (1/4,0) {$1$};
\end{tikzpicture}&=q^{-7/2}\begin{tikzpicture}[baseline=.5ex]
\draw [line width=1.5,->] (-1/2,0)--(1/2,0);
\end{tikzpicture} \label{spike}
\end{align}

\begin{remark}
This set of skein relations agrees with that in \cite{LS22} and is a simple renormalization of the set of skein relations in \cite{Hig23}. In particular, all results about $\Sq(\Sigma)$ proven in \cite{Hig23} carry over to the version used in this paper. The advantage to these coefficients is that the relations are symmetric with respect to the operation of reversing all arrows on strands. See also \cite{Kim22} for a similar renormalization. 
\end{remark}

A consequence of the defining relations which we will use later is the following.

\begin{proposition}
	The following holds for any $a,b \in \{1,2,3\}$ with $a>b.$
	
	\begin{equation}
		\begin{tikzpicture}[baseline=.5ex]
			\draw [line width=1.5, <-] (-1/4,0)--(3/4,0);
			\draw (0,0)--(1/4,1/2);
			\draw (1/2,0)--(1/4,1/2);
			\draw (1/4,1/2)--(1/4,1);
			\draw [-<] (0,0)--(1/8,1/4);
			\draw [-<] (1/2,0)--(3/8,1/4);
			\draw [-<] (1/4,1)--(1/4,3/4);
			\node [below] at (0,0) {$b$};
			\node [below] at (1/2,-.07) {$a$};
		\end{tikzpicture}=-q^{1/6}\begin{tikzpicture}[baseline=.5ex]
			\draw [line width=1.5, <-] (-1/4,0)--(3/4,0);
			\draw [->] (1/4,0)--(1/4,1/2);
			\draw (1/4,0)--(1/4,1);
			\node [below] at (1/4,0) {$a+b-2$};
		\end{tikzpicture}\label{left spike}
	\end{equation}
\end{proposition}

\begin{proof}
	We compute by first applying a height-preserving isotopy along the boundary and then apply the following relations:
	\begin{align*}
		\begin{tikzpicture}[baseline=.5ex]
			\draw [line width=1.5, <-] (-1/4,0)--(3/4,0);
			\draw (0,0)--(1/4,1/2);
			\draw (1/2,0)--(1/4,1/2);
			\draw (1/4,1/2)--(1/4,1);
			\draw [-<] (0,0)--(1/8,1/4);
			\draw [-<] (1/2,0)--(3/8,1/4);
			\draw [-<] (1/4,1)--(1/4,3/4);
			\node [below] at (0,0) {$b$};
			\node [below] at (1/2,-.07) {$a$};
		\end{tikzpicture}=
		\begin{tikzpicture}[baseline=.5ex]
			\draw [line width=1.5, ->] (-1/4,0)--(3/4,0);
			\node (center) at (1/4,1/4) {};
			\draw (1/4,1/2) to [out=-135, in=135] (1/2,0);
			\draw (1/4,1/2) to [out=-45,in=45] (center);
			\draw (center) -- (0,0);
			\draw (1/4,1/2)--(1/4,1);
			\draw [-<] (1/4,1)--(1/4,3/4);
			\node [below] at (0,-.07) {$a$};
			\node [below] at (1/2,0) {$b$};
		\end{tikzpicture}
		&\overset{(\ref{crossing})}{=}q^{-2/3}
		\begin{tikzpicture}[baseline=.5ex]
			\draw [line width=1.5, ->] (-1/4,0)--(3/4,0);
			\draw (0,0)--(1/4,1/2);
			\draw (1/2,0)--(1/4,1/2);
			\draw (1/4,1/2)--(1/4,1);
			\draw [-<] (0,0)--(1/8,1/4);
			\draw [-<] (1/2,0)--(3/8,1/4);
			\draw [-<] (1/4,1)--(1/4,3/4);
			\node [below] at (0,-.07) {$a$};
			\node [below] at (1/2,0) {$b$};
		\end{tikzpicture}
		+q^{1/3}
		\begin{tikzpicture}[baseline=.5ex]
			\draw [line width=1.5, ->] (-1/4,0)--(3/4,0);
			\draw (0,0) to [out=45, in=-90] (1/4,1/4);
			\draw (1/2,0) to [out=135,in=-90] (1/4,1/4);
			\draw (1/4,1/4) to (1/4,3/8);
			\draw (1/4,3/8) to [out=45, in =-45] (1/4,5/8);
			\draw (1/4,3/8) to [out=135, in =-135] (1/4,5/8);
			\draw (1/4,5/8) to (1/4, 1);
			\draw [-<] (1/4,1) to (1/4, 3/4);
			\node [below] at (0,-.07) {$a$};
			\node [below] at (1/2,0) {$b$};
		\end{tikzpicture}\\
		&\overset{(\ref{2gon})}{=}-q^{4/3}
		\begin{tikzpicture}[baseline=.5ex]
			\draw [line width=1.5, ->] (-1/4,0)--(3/4,0);
			\draw (0,0)--(1/4,1/2);
			\draw (1/2,0)--(1/4,1/2);
			\draw (1/4,1/2)--(1/4,1);
			\draw [-<] (0,0)--(1/8,1/4);
			\draw [-<] (1/2,0)--(3/8,1/4);
			\draw [-<] (1/4,1)--(1/4,3/4);
			\node [below] at (0,-.07) {$a$};
			\node [below] at (1/2,0) {$b$};
		\end{tikzpicture}
		\overset{(\ref{spike})}{=}-q^{1/6}
		\begin{tikzpicture}[baseline=.5ex]
			\draw [line width=1.5, <-] (-1/4,0)--(3/4,0);
			\draw [->] (1/4,0)--(1/4,1/2);
			\draw (1/4,0)--(1/4,1);
			\node [below] at (1/4,0) {$a+b-2$};
		\end{tikzpicture}
	\end{align*}
\end{proof}

\subsection{The splitting homomorphism}

We recall the splitting homomorphism for stated skein algebras of thickened surfaces, which is an algebra map associated to cutting a surface along an ideal arc. We call $c$ an \textit{ideal arc} on $\Sigma$ if it is a proper embedding
\begin{equation*}
c:(-1,1) \rightarrow \text{int}(\Sigma)
\end{equation*} such that its endpoints lie in $\text{cl}(\Sigma) \setminus \Sigma.$ Typically the endpoints of $c$ will be (not necessarily distinct) punctures of $\Sigma.$

Now suppose that $\Sigma$ is a surface, which is not necessarily connected, and has two distinct boundary arcs $a$ and $b.$ Let $\bar{\Sigma}=\Sigma / (a=b)$ be the oriented surface obtained by gluing $\Sigma$ along $a$ and $b$ in a way compatible with the orientation of $\Sigma.$ Denote by $g: \Sigma \rightarrow \bar{\Sigma}$ the gluing map. Then $c=g(a)=g(b)$ is an ideal arc on $\bar{\Sigma}.$ Extend the gluing map of surfaces $g: \Sigma \rightarrow \bar{\Sigma}$ to a gluing map for thickened surfaces in the obvious way, still denoted by $g: \Sigma \times I \rightarrow \bar{\Sigma} \times I.$

We define the splitting homomorphism $\Delta_c: \Sq(\bar{\Sigma}) \rightarrow \Sq(\Sigma)$ in the following manner on a stated web $W \subset \bar{\Sigma} \times I.$ Isotope the web $W$ so that $W \cap (c \times I)$ contains finitely many points where an edge of $W$ meets $c \times I$ transversely with vertical framing and so that the points of $W$ intersecting $c \times I$ occur at distinct heights. Let $W'$ be a lift of $W$ under the gluing so that $g(W')=W.$ For every endpoint $x' \in \partial W'$ for which $g(x') \in \partial W,$ the endpoint $x'$ inherits a state from $W$. For the endpoints of $W'$ created by cutting $W,$ we will sum over so-called admissible states. We say that $s'$ is an \textit{admissible state} of $W'$ if $s'(x)=s(g(x))$ for every $x \in g^{-1}(\partial W)$ and if for any $y,z \in g^{-1}(W \cap (c \times I))$ such that $g(y)=g(z)$ we have $s'(y)=s'(z).$

We then define the image of the splitting homomorphism $\Delta_c$ on the stated web $(W,s)$ to be

\begin{equation}
\Delta_c(W,s)=\sum_{\text{admissible} s'} [(W',s')]
\end{equation}

\begin{theorem}\label{splitting}
The following are satisfied by $\Delta_c.$
\begin{enumerate}
	\item The description of $\Delta_c$ above linearly extends to an algebra homomorphism \begin{equation*}\Delta_c: \Sq(\bar{\Sigma})\rightarrow \Sq(\Sigma).\end{equation*}
	\item For any ideal arc $c$ on a punctured bordered surface $\bar{\Sigma},$ the homomorphism $\Delta_c$ is injective. 
	\item If $c_1$ and $c_2$ are ideal arcs on $\bar{\Sigma}$ with disjoint interiors, then successive splittings can be applied in any order: \begin{equation*}\Delta_{c_1} \circ \Delta_{c_2}=\Delta_{c_2} \circ \Delta_{c_1}.\end{equation*}
	\item Suppose that $a$ and $b$ are the boundary arcs on $\Sigma$ satisfying $g(a)=g(b)=c$ under the gluing map. The image of $\Delta_c$ is equal to
	\begin{equation*}
	\textrm{im}(\Delta_c)=\textrm{ker}(\Delta_a-\tau \circ \Delta_b),
	\end{equation*} where $\tau$ is the tensor flip map.
\end{enumerate}	
\end{theorem}

The above theorem was proven in Theorems 4.2, 8.1, and 8.2 of \cite{Hig23} for $SL_3,$ and those theorems generalized the results in the $SL_2$ case described in \cite{Le18,KQ19,CL19}. In \cite{LS22}, the splitting homomorphism was defined for $SL_n$ stated skein algebras but the injectivity of the splitting homomorphism has not yet been proven in the higher rank case $n>3$ except for special cases \cite{LS22,Wang23,Wang24}.

\subsection{Stated skein algebras of small surfaces}

The splitting homomorphism from Theorem \ref{splitting} allows us to study the skein algebra of a punctured surface by splitting the surface into smaller pieces. The fundamental building-block surfaces are the monogon $\mathfrak{M},$ the bigon $\mathfrak{B},$ and the triangle $\mathfrak{T},$ which are given by a disk with one, two, or three points removed from its boundary, respectively. They are especially useful to us since their stated skein algebras are described easily in terms of the quantum group $\Oq,$ and consequently they have explicit presentations.

\begin{theorem}(\cite[Theorem 1.3]{Hig23})
\begin{itemize}
	\item $\mathcal{S}_q^{SL_3}(\mathfrak{M})=\mathcal{R}.$
	\item $\mathcal{S}_q^{SL_3}(\mathfrak{B})=\mathcal{O}_q(SL_3),$ as Hopf algebras.
	\item $\mathcal{S}_q^{SL_3}(\mathfrak{T})= \mathcal{O}_q(SL_3) \underset{-}{\otimes} \mathcal{O}_q(SL_3),$ the braided tensor square of $\mathcal{O}_q(SL_3).$
\end{itemize}
\end{theorem}

Often, we will view $\mathfrak{B}$ interchangeably as a disk with 2 points removed from its boundary or as $\mathfrak{B}=(-1,1) \times [0,1].$ It was shown in \cite{Hig23} that the algebra $\Sq(\mathfrak{B})$ has a generating set of single stated arcs oriented from the left side of the bigon to the right, depicted below. The isomorphism $\Sq(\mathfrak{B}) \cong \mathcal{O}_q(SL_3)$ is given by

\begin{equation}\label{bigon}
\begin{tikzpicture}[baseline=3ex]
	\draw [line width=1.5, rounded corners] (0,0)--(-1/2,1/4)--(-1/2,1/2+.1);
	\draw [line width=1.5, rounded corners] (-1/2,1/2)--(-1/2,3/4)--(0,1);
	\draw [line width=1.5, rounded corners]
	(0,0)--(1/2,1/4)--(1/2,1/2+.1);
	\draw [line width=1.5, rounded corners]
	(1/2,1/2)--(1/2,3/4)--(0,1);
	\filldraw [fill=white, draw=black] (0,0) circle [radius=.05];
	\filldraw [fill=white, draw=black] (0,1) circle [radius=.05];
	\draw [->] (-1/2,1/2)--(1/8,1/2);
	\draw (0,1/2)--(1/2,1/2);
	\node [left] at (-1/2,1/2) {$i$};
	\node [right] at (1/2,1/2) {$j$};
\end{tikzpicture} \mapsto X_{ij}.
\end{equation}

The coproduct for $\Sq(\mathfrak{B})$ is given by the splitting homomorphism associated to splitting along the ideal arc $c$ connecting the two punctures of $\mathfrak{B}$,

\begin{equation*}
\Delta_c: \Sq(\mathfrak{B}) \rightarrow \Sq(\mathfrak{B}) \otimes \Sq(\mathfrak{B}).
\end{equation*}

\subsection{Frobenius homomorphism for the bigon}\label{Frobenius bigon}

The stated skein algebra $\Sq(\mathfrak{B})$ inherits a Frobenius homomorphism from the Parshall-Wang homomorphism for $\Oq.$ When $q^{1/3}$ is a root of unity of order $N$ coprime to $6$ we have that $q$ and $q^2$ are both roots of unity of order $N$ and the map of Parshall-Wang $F_{\mathcal{O}_q}: \mathcal{O}_1(SL_3) \rightarrow \Oq$ from Theorem \ref{Frobenius Oq} exists. We then see that the corresponding map $F_{\mathfrak{B}}=F_{\mathcal{O}_q}$ given by the identification $\Sq(\mathfrak{B}) \cong \Oq$
\begin{equation*}
F_{\mathfrak{B}}: \Sc(\mathfrak{B}) \rightarrow \Sq(\mathfrak{B})
\end{equation*}
is defined on a simple stated arc appearing on the left side of (\ref{bigon}) by taking the $N\text{-th}$ power of the arc. In \cite[Proposition 12.5]{Hig23} it was further shown that the image of $F_{\mathfrak{B}}$ applied to a simple stated arc with the opposite orientation is also the $N\text{-th}$ power of that arc. Consequently, although the identification of $\Sq(\mathfrak{B})$ with $\Oq$ depends on a choice of left and right sides of $\mathfrak{B},$ the Frobenius homomorphism $F_{\mathfrak{B}}$ is invariant under this choice.

\subsection{Skein algebra of the annulus}

Another important small surface is the annulus $\mathcal{A}$ which we will view interchangeably as either a sphere with two punctures, or as $S^1 \times I.$

\begin{theorem}(\cite[Proposition 1.4]{OY97}) \label{OY}
Let $\mathcal{A}=S^1 \times I$ be the annulus with a choice of orientation of its core $S^1$, which we will call the positive direction. Let $l_+ \in \mathcal{S}_q^{SL_3}(\mathcal{A})$ represent a simple loop traversing the core of the annulus in the positive direction, and let $l_-$ represent the same loop with the opposite orientation. The skein algebra of the annulus is the following 2-variable polynomial algebra:
\begin{equation*}\mathcal{S}_q^{SL_3}(\mathcal{A})=\mathcal{R}[l_+,l_-].
\end{equation*}
\end{theorem}

The loops $l_+$ and $l_-$ are pictured below.

\begin{align*}
l_+&=\begin{tikzpicture}[scale=0.3, baseline = (B.base)]
	\def \linewidth{1.5};
	\def \radius{1.375};
	\def \innerradius{0.25};
	\coordinate (B) at (0,-0.25);
	\node (O) at (0,0) {};
	\draw [<-] (0,1.375) arc [radius=\radius, start angle=90, end angle=465];
	\draw[line width= \linewidth] (O) circle (\innerradius);
	\draw[line width= \linewidth] (O) circle (2.5);
\end{tikzpicture}&l_-&=
\begin{tikzpicture}[scale=0.3, baseline = (B.base)]
	\def \linewidth{1.5};
	\def \radius{1.375};
	\def \innerradius{0.25};
	\coordinate (B) at (0,-0.25);
	\node (O) at (0,0) {};
	\draw [->] (0,1.375) arc [radius=\radius, start angle=90, end angle=465];
	\draw[line width= \linewidth] (O) circle (\innerradius);
	\draw[line width= \linewidth] (O) circle (2.5);
\end{tikzpicture}
\end{align*}

The paper \cite{OY97} proved that the loops $l_+,l_-$ generate $\mathcal{S}_q^{SL_3}(\mathcal{A})$ by adapting the Euler characteristic argument of \cite{Kup94}. They prove the algebraic independence of the two loops $l_+,l_-$ in the appendix of their paper. An easier proof of the algebraic independence of the loops follows from an application of the confluence theory of \cite{SW07}.

Let $a$ be the ideal arc whose two endpoints are the two punctures of $\mathcal{A}.$ Splitting $\mathcal{A}$ along $a$ produces a bigon $\mathfrak{B}.$ In order to identify $\Sq(\mathfrak{B})$ with $\Oq$ we must choose a left and right side of $\mathfrak{B}.$ We make the choice so that $\Delta_a(l_+)=X_{11}+X_{22}+X_{33}$ after our identification. The diagram below depicts the splitting of $l_+$ under such a choice.

\begin{equation*}
\begin{tikzpicture}[scale=0.3, baseline = (B.base)]
	\def \linewidth{1.5};
	\def \radius{1.375};
	\def \innerradius{0.25};
	\coordinate (B) at (0,-0.25);
	\node (O) at (0,0) {};
	\draw [<-] (0,1.375) arc [radius=\radius, start angle=90, end angle=465];
	\draw[line width= \linewidth] (O) circle (\innerradius);
	\draw[line width= \linewidth] (O) circle (2.5);
	\draw[line width=1.25, dotted] (0,-2.5)--(O);
\end{tikzpicture} \overset{\Delta_a}{\mapsto}  \sum_{i=1}^3
\begin{tikzpicture}[baseline=3ex]
	\draw [line width=1.5, rounded corners] (0,0)--(-1/2,1/4)--(-1/2,1/2+.1);
	\draw [line width=1.5, rounded corners] (-1/2,1/2)--(-1/2,3/4)--(0,1);
	\draw [line width=1.5, rounded corners]
	(0,0)--(1/2,1/4)--(1/2,1/2+.1);
	\draw [line width=1.5, rounded corners]
	(1/2,1/2)--(1/2,3/4)--(0,1);
	\filldraw [fill=white, draw=black] (0,0) circle [radius=.05];
	\filldraw [fill=white, draw=black] (0,1) circle [radius=.05];
	\draw [->] (-1/2,1/2)--(1/8,1/2);
	\draw (0,1/2)--(1/2,1/2);
	\node [left] at (-1/2,1/2) {$i$};
	\node [right] at (1/2,1/2) {$i$};
\end{tikzpicture}
\end{equation*}

\begin{proposition}\label{splitting annulus}
Applying the splitting map $\Delta_a: \mathcal{S}_q^{SL_3}(\mathcal{A}) \rightarrow \mathcal{S}_q^{SL_3}(\mathfrak{B}) \cong \mathcal{O}_q(SL_3)$ to the loops $l_+,l_-$ yields the following results.
\begin{align*}
	\Delta_a(l_+)=\sigma_1:=&X_{11}+X_{22}+X_{33} \in \mathcal{O}_q(SL_3)\\
	\Delta_a(l_-)=\sigma_2:=&X_{11}X_{22}-qX_{12}X_{21}\\
	&+X_{22}X_{33}-qX_{23}X_{32}\\
	&+X_{11}X_{33}-qX_{13}X_{31} \in \mathcal{O}_q(SL_3),
\end{align*}under our identification of $\mathcal{S}_q^{SL_3}(\mathfrak{B})$ with $\mathcal{O}_q(SL_3).$
\end{proposition}

\begin{proof}
The first equality follows directly from the choice we made for the left and right side of $\mathfrak{B}.$ For the second equality, once we apply the splitting map, we get a sum of 3 diagrams, each consisting of a stated strand with orientation opposite to our chosen generators of $\Oq.$

To rewrite each diagram in terms of the generators of $\mathcal{O}_q(SL_3)$ we apply some relations in the following way. After applying relation (\ref{spike}) to the right boundary edge of the bigon and then applying relation (\ref{left spike}) to the left boundary edge of the bigon we obtain the following.

\begin{align*}
\begin{tikzpicture}[baseline=3ex]
\draw [line width=1.5, rounded corners] (0,0)--(-1/2,1/4)--(-1/2,1/2+.1);
\draw [line width=1.5, rounded corners] (-1/2,1/2)--(-1/2,3/4)--(0,1);
\draw [line width=1.5, rounded corners]
(0,0)--(1/2,1/4)--(1/2,1/2+.1);
\draw [line width=1.5, rounded corners]
(1/2,1/2)--(1/2,3/4)--(0,1);
\filldraw [fill=white, draw=black] (0,0) circle [radius=.05];
\filldraw [fill=white, draw=black] (0,1) circle [radius=.05];
\draw [-<] (-1/2,1/2)--(1/8,1/2);
\draw (0,1/2)--(1/2,1/2);
\node [left] at (-1/2,1/2) {$1$};
\node [right] at (1/2,1/2) {$1$};
\end{tikzpicture}
&\overset{(\ref{spike})}{=}q^{7/6}
\begin{tikzpicture}[baseline=3ex]
\draw [line width=1.5, rounded corners,->] (0,0)--(-1/2,1/4)--(-1/2,1/2+.1);
\draw [line width=1.5, rounded corners] (-1/2,1/2)--(-1/2,3/4)--(0,1);
\draw [line width=1.5, rounded corners,->] (0,0)--(1/2,1/4)--(1/2,1/2+.1);
\draw [line width=1.5, rounded corners] (1/2,1/2)--(1/2,3/4)--(0,1);
\draw (-1/2,1/2)--(0,1/2);
\draw [-<] (-1/2,1/2)--(-1/4,1/2);
\draw (0,1/2)--(1/2,3/4);
\draw [->] (0,1/2)--(1/4,5/8);
\draw(0,1/2)--(1/2,1/4);
\draw[->] (0,1/2)--(1/4,3/8);
\filldraw [fill=white, draw=black] (0,0) circle [radius=.05];
\filldraw [fill=white, draw=black] (0,1) circle [radius=.05];
\node [left] at (-1/2-.1,1/2) {$1$};
\node [right] at (1/2,3/4) {$1$};
\node [right] at (1/2,1/4) {$2$};
\end{tikzpicture}
\overset{(\ref{left spike})}{=}-q
\begin{tikzpicture}[baseline=3ex]
\draw [line width=1.5, rounded corners,->] (0,0)--(-1/2,1/4)--(-1/2,1/2+.1);
\draw [line width=1.5, rounded corners] (-1/2,1/2)--(-1/2,3/4)--(0,1);
\draw [line width=1.5, rounded corners,->] (0,0)--(1/2,1/4)--(1/2,1/2+.1);
\draw [line width=1.5, rounded corners] (1/2,1/2)--(1/2,3/4)--(0,1);
\draw (-1/2,3/4)--(-1/4,1/2);
\draw (-1/2,1/4)--(-1/4,1/2);
\draw (1/2,3/4)--(1/4,1/2);
\draw (1/2,1/4)--(1/4,1/2);
\draw (-1/4,1/2)--(1/4,1/2);
\draw [-<](-1/4,1/2)--(0,1/2);
\filldraw [fill=white, draw=black] (0,0) circle [radius=.05];
\filldraw [fill=white, draw=black] (0,1) circle [radius=.05];
\node [left] at (-1/2,3/4) {$1$};
\node [left] at (-1/2,1/4) {$2$};
\node [right] at (1/2,3/4) {$1$};
\node [right] at (1/2, 1/4) {$2$};
\end{tikzpicture}\\
&\overset{(\ref{switch})}{=}-q\bigg(-q^{-1}
\begin{tikzpicture}[baseline=3ex]
\draw [line width=1.5, rounded corners,->] (0,0)--(-1/2,1/4)--(-1/2,1/2+.1);
\draw [line width=1.5, rounded corners] (-1/2,1/2)--(-1/2,3/4)--(0,1);
\draw [line width=1.5, rounded corners,->] (0,0)--(1/2,1/4)--(1/2,1/2+.1);
\draw [line width=1.5, rounded corners] (1/2,1/2)--(1/2,3/4)--(0,1);
\draw (-1/2,3/4)--(1/2,3/4);
\draw [->] (-1/2,3/4)--(0,3/4);
\draw (-1/2,1/4)--(1/2,1/4);
\draw [->] (-1/2,1/4)--(0,1/4);
\filldraw [fill=white, draw=black] (0,0) circle [radius=.05];
\filldraw [fill=white, draw=black] (0,1) circle [radius=.05];
\node [left] at (-1/2,3/4) {$1$};
\node [left] at (-1/2,1/4) {$2$};
\node [right] at (1/2,3/4) {$1$};
\node [right] at (1/2, 1/4) {$2$};
\end{tikzpicture}
+
\begin{tikzpicture}[baseline=3ex]
\draw [line width=1.5, rounded corners,->] (0,0)--(-1/2,1/4)--(-1/2,1/2+.1);
\draw [line width=1.5, rounded corners] (-1/2,1/2)--(-1/2,3/4)--(0,1);
\draw [line width=1.5, rounded corners,->] (0,0)--(1/2,1/4)--(1/2,1/2+.1);
\draw [line width=1.5, rounded corners] (1/2,1/2)--(1/2,3/4)--(0,1);
\draw (-1/2,3/4)--(1/2,3/4);
\draw [->] (-1/2,3/4)--(0,3/4);
\draw (-1/2,1/4)--(1/2,1/4);
\draw [->] (-1/2,1/4)--(0,1/4);
\filldraw [fill=white, draw=black] (0,0) circle [radius=.05];
\filldraw [fill=white, draw=black] (0,1) circle [radius=.05];
\node [left] at (-1/2,3/4) {$1$};
\node [left] at (-1/2,1/4) {$2$};
\node [right] at (1/2,3/4) {$2$};
\node [right] at (1/2, 1/4) {$1$};
\end{tikzpicture}\bigg)\\
&=X_{11}X_{22}-qX_{12}X_{21},	
\end{align*}

where final equality comes from our identification of $\Sq(\mathfrak{B}) \cong \mathcal{O}_q(SL_3)$ given by (\ref{bigon}).

Similarly, we compute that

\begin{align*}
\begin{tikzpicture}[baseline=3ex]
\draw [line width=1.5, rounded corners] (0,0)--(-1/2,1/4)--(-1/2,1/2+.1);
\draw [line width=1.5, rounded corners] (-1/2,1/2)--(-1/2,3/4)--(0,1);
\draw [line width=1.5, rounded corners]
(0,0)--(1/2,1/4)--(1/2,1/2+.1);
\draw [line width=1.5, rounded corners]
(1/2,1/2)--(1/2,3/4)--(0,1);
\filldraw [fill=white, draw=black] (0,0) circle [radius=.05];
\filldraw [fill=white, draw=black] (0,1) circle [radius=.05];
\draw [-<] (-1/2,1/2)--(1/8,1/2);
\draw (0,1/2)--(1/2,1/2);
\node [left] at (-1/2,1/2) {$2$};
\node [right] at (1/2,1/2) {$2$};
\end{tikzpicture}&=X_{11}X_{33}-qX_{13}X_{31} \text{ and }\\
\begin{tikzpicture}[baseline=3ex]
\draw [line width=1.5, rounded corners] (0,0)--(-1/2,1/4)--(-1/2,1/2+.1);
\draw [line width=1.5, rounded corners] (-1/2,1/2)--(-1/2,3/4)--(0,1);
\draw [line width=1.5, rounded corners]
(0,0)--(1/2,1/4)--(1/2,1/2+.1);
\draw [line width=1.5, rounded corners]
(1/2,1/2)--(1/2,3/4)--(0,1);
\filldraw [fill=white, draw=black] (0,0) circle [radius=.05];
\filldraw [fill=white, draw=black] (0,1) circle [radius=.05];
\draw [-<] (-1/2,1/2)--(1/8,1/2);
\draw (0,1/2)--(1/2,1/2);
\node [left] at (-1/2,1/2) {$3$};
\node [right] at (1/2,1/2) {$3$};
\end{tikzpicture}&=X_{22}X_{33}-qX_{23}X_{31}.
\end{align*}

We conclude that

\begin{equation*}
\Delta(l_-)=\sum_{i=1}^3 \begin{tikzpicture}[baseline=3ex]
\draw [line width=1.5, rounded corners] (0,0)--(-1/2,1/4)--(-1/2,1/2+.1);
\draw [line width=1.5, rounded corners] (-1/2,1/2)--(-1/2,3/4)--(0,1);
\draw [line width=1.5, rounded corners]
(0,0)--(1/2,1/4)--(1/2,1/2+.1);
\draw [line width=1.5, rounded corners]
(1/2,1/2)--(1/2,3/4)--(0,1);
\filldraw [fill=white, draw=black] (0,0) circle [radius=.05];
\filldraw [fill=white, draw=black] (0,1) circle [radius=.05];
\draw [-<] (-1/2,1/2)--(1/8,1/2);
\draw (0,1/2)--(1/2,1/2);
\node [left] at (-1/2,1/2) {$i$};
\node [right] at (1/2,1/2) {$i$};
\end{tikzpicture}=\sigma_2,
\end{equation*}

and thus $\Delta_a(l_-)$ becomes the sum of the principle quantum $2\text{-by-}2$ minors in $\mathcal{O}_q(SL_3),$ as claimed.

\end{proof}

\section{The Frobenius map for $SL_3$ skein modules}

The goal of this section is to show that the threading operation from \cite{BH23} respects the skein relations and defines a Frobenius homomorphism for $SL_3$ skein modules. The key step will be to connect the threading operation to the Frobenius homomorphism defined in \cite{Hig23} in the case of punctured surfaces.

\subsection{Threading and power sum polynomials}

The threading operations described in \cite{BH23} were in the context of the CKM webs of Cautis-Kamnitzer-Morrison \cite{CKM14}, whose edges carry integer labels in $\{1,2\}$ in addition to orientations. These webs are well known to be equivalent to Kuperberg webs in the $SL_3$ case. To translate from the setting of $SL_3$ CKM webs to Kuperberg webs, one reverses the orientation of 2-labeled edges, then forgets edge labels and removes tags. The following describes the threading operation for links in the $SL_3$ setting of Kuperberg webs.

\begin{definition}
	Suppose that $K$ is an oriented, framed knot in an oriented 3-manifold and that $e_1^ie_2^j \in \mathcal{R}[e_1,e_2]$ is a monomial in 2 variables. The threading of $e_1^ie_2^j$ along $K$ is denoted by $K^{[e_1^ie_2^j]}$ and is given by taking $i+j$ parallel copies of $K$ in the direction of the framing, and giving $i$ of these the same orientation as $K$ and giving $j$ of these the opposite orientation of $K.$ Extending this construction linearly, the threading of a polynomial $P=\sum_{i,j} a_{ij}e_1^ie_2^j \in \mathcal{R}[e_1,e_2]$ along $K$ is denoted by $K^{[P]}$ and is given by $K^{[P]}=\sum a_{ij}K^{[e_1^ie_2^j]}.$
	
	The threading of polynomials $P_1,...,P_k \in \mathcal{R}[e_1,e_2]$ along an oriented, framed link consisting of $k$ components $K_1,...,K_k$ is defined by threading each $P_i$ along $K_i$ in the multilinear fashion. 
\end{definition}

\begin{definition}\label{power sum}
	Consider the ring $\Lambda=\mathcal{R}[\lambda_1,\lambda_2,\lambda_3]/(\lambda_1\lambda_2\lambda_3-1).$ Consider the two elements in this ring defined by $E_1=\lambda_1+\lambda_2+\lambda_3$ and $E_2=\lambda_1\lambda_2+\lambda_1\lambda_3+\lambda_2\lambda_3.$ Note that the elements $E_1$ and $E_2$ are algebraically independent in $\Lambda.$ We define the power sum polynomial $P^{(N)} \in \mathcal{R}[e_1,e_2]$ to be the unique polynomial in two variables satisfying the property that \begin{equation*}
		P^{(N)}(E_1,E_2)=\lambda_1^N+\lambda_2^N+\lambda_3^N.
	\end{equation*}
\end{definition}

\begin{remark}
	The polynomial $P^{(N)}$ coincides with the polynomial described by the slightly different notation $\hat{P}_3^{(N,1)}$ in \cite{BH23}. The second power elementary polynomial $\hat{P}_3^{(N,2)}(e_1,e_2)$ from \cite{BH23} corresponds to swapping $e_2$ with $e_1$ in $P^{(N)}$ to obtain $P^{(N)}(e_2,e_1)$. This symmetry in the two power elementary polynomials arises from the fact that the two fundamental representations for $SL_3$ are dual to each other.
\end{remark}

For a concrete construction of $P^{(N)},$ it can be defined using the following recursive definition, which shows each polynomial has integer coefficients:

\begin{align*}
	P^{(0)}&=3\\
	P^{(1)}&=e_1\\
	P^{(2)}&=e_1^2-2e_2\\
	P^{(N)}&=e_1P^{(N-1)}-e_2P^{(N-2)}+P^{(N-3)}, \text{ for } N\geq 3.
\end{align*}

Recall that in the setting of $SL_3$ skein modules, the main theorem of \cite{BH23} showed that threadings of power sum polynomials along links produce \textit{transparent} elements in the skein module $\Sq(M).$ This means that the threading operation of a power sum polynomial of a knot respects any isotopy of the knot in $M$ which is allowed to pass through any other skein in $M$.

Because of relations $(\ref{2gon})$ and $(\ref{crossing}),$ the set of oriented links in an oriented 3-manifold $M$ forms a spanning set for $\Sq(M).$ We will work towards proving the following theorem, which says that the threading operation along links respects the skein relations and so defines a skein module homomorphism.

\begin{theorem}\label{main theorem}
	Suppose that $M$ is an oriented 3-manifold. When $q^{1/3} \in \mathcal{R}$ is a root of unity of order $N$ coprime to $6,$ then the threading operation
	\begin{equation*}
	L \mapsto L^{[P^{(N)}]}
	\end{equation*}
	that sends an oriented link $L$ in $M$ to the result of threading $P^{(N)}$ along each component of the link defines a homomorphism of skein modules $\mathcal{S}_1^{SL_3}(M) \rightarrow \mathcal{S}_q^{SL_3}(M).$ Furthermore, in the particular case when $M=\Sigma \times I$ is a thickened oriented surface, then the threading operation defines an algebra homomorphism of skein algebras $\mathcal{S}_1^{SL_3}(\Sigma) \rightarrow Z(\mathcal{S}_q^{SL_3}(\Sigma)).$ In case the surface has at least one puncture on each connected component, the resulting algebra map coincides with the Frobenius map defined in \cite{Hig23} and is thus known to be injective.
\end{theorem}

\begin{remark}
	We expect that the threading map is injective in the case of closed surfaces as well, although we imagine it will require the development of new tools to prove this, analogous to those developed in \cite{FKBL19}. On the other hand, we expect that the threading map for 3-manifolds will not always be injective, as was the case for $SL_2$ \cite{CL22}.
\end{remark}

The proof of Theorem \ref{main theorem} will be given by Corollary \ref{Frobenius link}, Theorem \ref{Frobenius surface}, and Theorem \ref{Frobenius manifold}.

We will begin our investigation with the most fundamental case, which is the computation of the Frobenius homomorphism on a simple loop around the annulus. Afterwards, we will spend the rest of this section leveraging this case to give the general case.

\subsection{Miraculous cancellations and the Frobenius image of a loop in the annulus}

Our first goal is to prove the theorem in the case of punctured surfaces. Our first step toward the goal of relating the threading operation from \cite{BH23} to the Frobenius map from \cite{Hig23} will be to compute it in the special case that $\Sigma$ is the 2-punctured sphere or annulus, $\Sigma = \mathcal{A}.$

Recall the Frobenius map for the bigon from Section \ref{Frobenius bigon}. Consider the splitting map $\Delta_a: \mathcal{S}_q^{SL_3}(\mathcal{A}) \rightarrow \mathcal{S}_q^{SL_3}(\mathfrak{B})$ associated to splitting the annulus $\mathcal{A}$ along an ideal arc traveling from one puncture of $\mathcal{A}$ to the other. From Theorem \ref{splitting}, the map $\Delta_a$ fits into an exact sequence. Using that exact sequence, the Frobenius map $F_{\mathcal{A}}$ for the annulus was constructed in \cite[Proposition 12.12]{Hig23} as the map making the following square of algebras commute.

\begin{equation}
	\begin{tikzcd}
		\mathcal{S}_1^{SL_3}(\mathcal{A}) \arrow[r,"\Delta_a"] \arrow[d,dashed, "F_{\mathcal{A}}"] &  \mathcal{S}_1^{SL_3}(\mathfrak{B}) \arrow[d,"F_{\mathfrak{B}}"] \\
		\mathcal{S}_q^{SL_3}(\mathcal{A}) \arrow[r,"\Delta_a"] & \mathcal{S}_q^{SL_3}(\mathfrak{B}) \\
	\end{tikzcd} \label{commutative annulus}
\end{equation}

By Theorem \ref{OY}, the skein algebra of the annulus (for any specialization of $q$) is isomorphic to $\mathcal{R}[l_+,l_-],$ a polynomial algebra in the two simple oriented loops traveling around the core of the annulus. Also, the computation of Proposition \ref{splitting annulus} showed that

\begin{align*}
	\Delta_a(l_+)=\sigma_1:=&X_{11}+X_{22}+X_{33} \in \mathcal{O}_q(SL_3)\\
	\Delta_a(l_-)=\sigma_2:=&X_{11}X_{22}-qX_{12}X_{21}\\
	&+X_{22}X_{33}-qX_{23}X_{32}\\
	&+X_{11}X_{33}-qX_{13}X_{31} \in \mathcal{O}_q(SL_3),
\end{align*}under our identification of $\mathcal{S}_q^{SL_3}(\mathfrak{B})$ with $\mathcal{O}_q(SL_3).$

The following proposition now provides us with an explicit description of the image of $F_{\mathcal{A}}$ on a loop in terms of the threading of the polynomial $P^{(N)}$ along the loop.

\begin{proposition}\label{Frobenius loop}
	When $q^{1/3}$ is a root of unity of order $N$ coprime to $6,$ the image of a simple loop in the annulus under the Frobenius map $F_{\mathcal{A}}$ is equal to the threading by $P^{(N)}$ along the loop:
	
	\begin{align*}
		F_{\mathcal{A}}(l_+)&=P^{(N)}(l_+,l_-)\\
		F_{\mathcal{A}}(l_-)&=P^{(N)}(l_-,l_+). 
	\end{align*}
\end{proposition}

\begin{proof}
We first note that the second identity in the proposition is the same as the first with all arrows reversed. Since the definining skein relations (\ref{2gon})-(\ref{spike}) are invariant under the operation of reversing all arrows, it will suffice to prove the first identity.

Since $\mathcal{S}_q^{SL_3}(\mathcal{A})=\mathcal{R}[l_+,l_-],$ we know that $F_\mathcal{A}(l_+)=Q(l_+,l_-) \in \mathcal{R}[l_+,l_-]$ for some polynomial $Q.$ Our goal is to show that  $Q=P^{(N)}.$
	
From the commutative square (\ref{commutative annulus}), we get
	
	\begin{align*}
		F_\mathfrak{B} \circ \Delta_a (l_+)&= F_\mathfrak{B}(X_{11}+X_{22}+X_{33})\\
		&= X_{11}^N+X_{22}^N+X_{33}^N.
	\end{align*}
	
	On the other hand, we consider the computation of
	
	\begin{align*}
		\Delta_{a} \circ F_{\mathcal{A}}(l_+)&= \Delta_a(Q(l_+,l_-))\\
		&=Q(\Delta_a(l_+),\Delta_a(l_-))\\
		&=Q(\sigma_1,\sigma_2),
	\end{align*}

where the second equality uses the fact that $\Delta_a$ is an algebra homomorphism and $Q$ is a polynomial.

By the commutativity of the square, we have the following equality in $\mathcal{O}_q(SL_3):$

\begin{equation}\label{poly P}
	Q(\sigma_1,\sigma_2)=X_{11}^N+X_{22}^N+X_{33}^N.
\end{equation}

Now consider the diagonalization map

\begin{align*}
	D: \mathcal{O}_q(SL_3)&\rightarrow \mathcal{R}[\lambda_{1},\lambda_{2},\lambda_{3}]/(\lambda_{1}\lambda_{2}\lambda_{3}=1)\\
	X_{ij} & \mapsto \delta_{ij}\lambda_{i},
\end{align*}

which is seen to be an algebra homomorphism by checking that it respects each of the defining relations of $\mathcal{O}_q(SL_3)$ from Definition \ref{Oq}. Let $E_1=\lambda_1+\lambda_2+\lambda_3$ and let $E_2=\lambda_1\lambda_2 +\lambda_2\lambda_3+\lambda_1\lambda_3.$ Note the identities $D(\sigma_1)=E_1$ and $D(\sigma_2)=E_2.$

Applying $D$ to both sides of Equation (\ref{poly P}) yields

\begin{align*}
	D(Q(\sigma_1,\sigma_2)) &=D(X_{11}^N+X_{22}^N+X_{33}^N)\\
	Q(D(\sigma_1),D(\sigma_2))&=D(X_{11})^N+D(X_{22})^N+D(X_{33})^N\\
	Q(E_1,E_2)&=\lambda_1^N+\lambda_2^N+\lambda_3^N,
\end{align*}

where the second line uses the fact that $D$ is an algebra homomorphism and $Q$ is a polynomial.

By the defining property of $P^{(N)}$ from Definition \ref{power sum}, we see that
\begin{equation*}
Q(E_1,E_2)=P^{(N)}(E_1,E_2).
\end{equation*}

By the algebraic independence of $E_1$ and $E_2,$ we have that $Q=P^{(N)}.$

\end{proof}

\begin{remark}
The equality (\ref{poly P}) along with the deduction that $Q=P^{(N)}$ gives the equality $P^{(N)}(\sigma_1,\sigma_2)=X_{11}^N+X_{22}^N+X_{33}^N,$ which is an $SL_3$ analogue of the `miraculous cancellations' found in \cite{BW16,Bon19}, and generalizes the equality $T_N(a+d)=a^N+d^N$ in $ \mathcal{O}_q(SL_2)$ proven in \cite{KQ19,BL22}.
\end{remark}

\begin{corollary}\label{miraculous cancellations}
When $q^{1/3}$ is a root of unity of order $N$ coprime to $6$ we have the following identities in $\mathcal{O}_q(SL_3):$
\begin{align*}
	P^{(N)}(\sigma_1,\sigma_2)=&X_{11}^N+X_{22}^N+X_{33}^N\\
	P^{(N)}(\sigma_2,\sigma_1)=&X_{11}^NX_{22}^N+X_{11}^NX_{33}^N+X_{22}^NX_{33}^N\\
	&- X_{12}^NX_{21}^N-X_{13}^NX_{31}^N-X_{23}^NX_{32}^N.
\end{align*}
\end{corollary}

\begin{proof}
As remarked, the first identity appeared in the proof of Proposition \ref{Frobenius annulus}. The second identity can be seen by repeating the proof but applying both sides of (\ref{commutative annulus}) to $l_-$ instead of to $l_+.$
\end{proof}

\begin{corollary}\label{Frobenius annulus}
	Consider $k$ parallel ideal arcs with disjoint interiors connecting one puncture of $\mathcal{A}$ to the other. Denote by
	
	\begin{equation*}
		\Delta^{(k)}: \mathcal{S}_q^{SL_3}(\mathcal{A}) \rightarrow \mathcal{S}_q^{SL_3}(\mathfrak{B})^{\otimes k}
	\end{equation*}
	
	the result of successively applying the splitting map to each of the $k$ parallel arcs.
	
	Then the following identity holds:
	
	\begin{equation*}
		\Delta^{(k)}\circ F_\mathcal{A}= (F_{\mathfrak{B}}^{\otimes k}) \circ \Delta^{(k)}, 
	\end{equation*}

	and consequently
	
	\begin{equation*}
		\Delta^{(k)}F_{\mathcal{A}}(l_+)=\sum_{i,i_1,...,i_{k-1}} X_{ii_1}^N \otimes X_{i_1i_2}^N \otimes \cdots \otimes X_{i_{k-1}i}^N.
	\end{equation*}
	
\end{corollary}

\begin{proof}
	
	Since the maps involved are algebra maps, it will suffice to check the identity on the generators $l_+,l_-.$ Since the identity is symmetric with respect to reversing orientations of strands, we only need to check it on $l_+.$
	
	By the proof of Proposition \ref{Frobenius loop} we have $\Delta_a F_{\mathcal{A}}(l_+)=\Delta_a P^{(N)}(l_+,l_-)=F_{\mathfrak{B}}(\Delta_a(l_+)).$ Then using the fact that $F_{\mathfrak{B}}$ is a coalgebra homomorphism (Theorem \ref{Frobenius Oq}), we obtain that
	
	\begin{align*}
		\Delta^{(k)}F_{\mathcal{A}}(l_+)&= \Delta^{(k-1)}(\Delta_a F_{\mathcal{A}}(l_+))\\
		&= \Delta^{(k-1)}F_{\mathfrak{B}}(\Delta_a(l_+))\\
		&= F_{\mathfrak{B}}^{\otimes k}\Delta^{(k-1)}\Delta_a(l_+)\\
		&=F_{\mathfrak{B}}^{\otimes k}\Delta^{(k)}(l_+)\\
		&= F_{\mathfrak{B}}^{\otimes k}(\sum_{i,i_1,...,i_{k-1}}X_{ii_1} \otimes X_{i_1i_2} \otimes \cdots \otimes X_{i_{k-1}i})\\
		&= \sum_{i,i_1,...,i_{k-1}} X_{ii_1}^N \otimes X_{i_1i_2}^N \otimes \cdots \otimes X_{i_{k-1}i}^N.
	\end{align*}
	
\end{proof}

\subsection{The Frobenius map for punctured surfaces}

Suppose that $\Sigma$ is a punctured surface with an ideal triangulation. Then as shown in \cite{KQ19,Hig23}, by decomposing the surface into ideal triangles, we can define a Frobenius map on the skein algebra by using the relationship between stated skein algebras and quantum groups.

\begin{theorem}(\cite{Hig23}, see also \cite{KQ19, BL22} for $n=2$)\label{Frobenius triangulation}
Suppose that $q^{1/3} \in \mathcal{R}$ is a root of unity of order coprime to $6.$ If $\Sigma$ is a punctured surface so that each connected component of $\Sigma$ contains at least one puncture, then there exists an algebra embedding of $\mathcal{S}_1^{SL_3}(\Sigma)$ into the center of the skein algebra

\begin{equation*}
F_\Sigma: \mathcal{S}_1^{SL_3}(\Sigma) \hookrightarrow Z(\mathcal{S}_q^{SL_3}(\Sigma)),
\end{equation*}

induced by the Frobenius homomorphism for the quantum group $\Oq.$
\end{theorem}

For the $SL_2$ case the analogous construction agrees with the threading map of Bonahon-Wong \cite{BW16}. The goal of this section is to show that for $SL_3$ the construction agrees with the threading operation in \cite{BH23}.

\begin{definition}\label{Frobenius definition}
For any ideal triangulable surface, with a fixed triangulation consisting of edge set $\mathcal{E},$ the map $F_\Sigma$ is defined  in \cite[Section 12]{Hig23} to be the unique algebra homomorphism making the following diagram of algebras commute

\begin{equation*}
	\begin{tikzcd}
		\mathcal{S}_1^{SL_3}(\Sigma) \arrow[r,"\Delta_{\mathcal{E}}"] \arrow[d,dashed, "F_{\Sigma}"] & \displaystyle\bigotimes_{i=1}^{n} \mathcal{S}_1^{SL_3}(\mathfrak{T}_i)  \arrow[d,"\otimes_i F_{\mathfrak{T}_i}"] \\
		\mathcal{S}_q^{SL_3}(\Sigma) \arrow[r,"\Delta_{\mathcal{E}}"] & \displaystyle\bigotimes_{i=1}^{n} \mathcal{S}_q^{SL_3}(\mathfrak{T}_i).\\
	\end{tikzcd},
\end{equation*}
where $\Delta_\mathcal{E}$ is given by the successive application of the splitting maps for each ideal arc of the triangulation.
\end{definition}

The existence and uniqueness relies on applications of the exact sequence associated to the splitting homomorphism in Theorem \ref{splitting}. The maps $F_\mathfrak{T}$ are defined by $F_\mathfrak{T}=F_\mathfrak{B} \otimes F_\mathfrak{B}$ where $F_\mathfrak{B}: \Sc(\mathfrak{B}) \rightarrow \Sq(\mathfrak{B})$ is identified with the Frobenius map of Parshall-Wang using the identification $\mathcal{S}_q^{SL_3}(\mathfrak{B})\cong \mathcal{O}_q(SL_3).$

Although the above construction of $F_\Sigma$ guarantees the existence of the map and proves several nice properties about it, the construction does not immediately provide a topological description of the image of $F_\Sigma$ on a link $L$ in $\Sigma \times I.$ To build up to this, we will first examine the image of a collection of stated arcs in the triangle.

\subsection{The Frobenius image of a tangle in the triangle}

Let $T$ be a stated tangle, with no component which is a knot, in the thickened triangle $\mathfrak{T} \times I.$ The tangle $T$ is a union of stated framed arcs $\alpha_i.$ Denote by $\alpha_i^{(N)}$ the \textit{framed power} of the arc $\alpha_i,$ which is the result of taking the union of $N$ disjoint parallel copies of the stated arc $\alpha_i$ in the direction of the framing. The following can be deduced without much work from the moves derived in \cite[Section 12]{Hig23}, but Wang in \cite{Wang23} has also shown the analogous result for $SL_n$ using moves that we find convenient to cite here.

\begin{proposition}(\cite{Wang23})\label{Frobenius triangle}
For $T=\cup \alpha_i$ a tangle that is a union of stated arcs in $\Sc(\mathfrak{T}),$ the Frobenius image $F_\mathfrak{T}(T)$ is the tangle $F_{\mathfrak{T}}(T)=\cup \alpha_i^{(N)}$ consisting of framed powers of the stated arcs comprising $T.$
\end{proposition}

For completeness, we will recapitulate the details in our setting. 

\begin{proof}
When $q=1,$ the skein relations imply that stated arcs in the tangle $T$ satisfy special moves including the following crossing switch and height exchange moves, which hold for any consistent strand orientations and choice of states $a,b,c,d \in \{1,2,3\}$

\begin{align}
\label{q=1 moves}
\begin{tikzpicture}[baseline=3ex]
\node (center) at (0,1/2) {};
\draw [line width=1.5] (-1/2,-1/8)--(-1/2,1/8);
\draw [line width=1.5] (-1/2,1-1/8)--(-1/2,1+1/8);
\draw [line width=1.5] (1/2,-1/8)--(1/2,1/8);
\draw [line width=1.5] (1/2,1-1/8)--(1/2,1+1/8);
\draw (-1/2,0)--(center);
\draw (center)--(1/2,1);
\draw (-1/2,1)--(1/2,0);
\node [left] at (-1/2,1) {$a$};
\node [right] at (1/2,1) {$b$};
\node [left] at (-1/2,0) {$c$};
\node [right] at (1/2,0) {$d$};
\end{tikzpicture}&=
\begin{tikzpicture}[baseline=3ex]
	\node (center) at (0,1/2) {};
	\draw [line width=1.5] (-1/2,-1/8)--(-1/2,1/8);
	\draw [line width=1.5] (-1/2,1-1/8)--(-1/2,1+1/8);
	\draw [line width=1.5] (1/2,-1/8)--(1/2,1/8);
	\draw [line width=1.5] (1/2,1-1/8)--(1/2,1+1/8);
	\draw (-1/2,0)--(1/2,1);
	\draw (-1/2,1)--(center);
	\draw (1/2,0) -- (center);
	\node [left] at (-1/2,1) {$a$};
	\node [right] at (1/2,1) {$b$};
	\node [left] at (-1/2,0) {$c$};
	\node [right] at (1/2,0) {$d$};
\end{tikzpicture}
&	
\begin{tikzpicture}[baseline=3ex]
	\draw [line width =1.5, ->] (-1/4,0)--(3/4,0);
	\node (center) at (1/4,1/2) {};
	\draw (0,1) -- (1/2,0);
	\draw (1/2,1) -- (center);
	\draw (center)--(0,0);
	\node [below] at (0,-.07) {$a$};
	\node [below] at (1/2,0) {$b$};
\end{tikzpicture}&=\begin{tikzpicture}[baseline=3ex]
\draw [line width =1.5, ->] (-1/4,0)--(3/4,0);
\draw (0,1)--(0,0);
\draw (1/2,1)--(1/2,0);
\draw (0,1)--(0,1/2);
\draw  (1/2,1)--(1/2,1/2);
\node [below] at (0,0) {$b$};
\node [below] at (1/2,-.07) {$a$};
\end{tikzpicture},
\end{align}

along with (unframed) Reidemeister 1, as well as Reidemeister 2 and 3 as usual.

Consequently, when $q=1,$ the tangle $T \in \Sc(\mathfrak{T})$ can then be written as a product of stated simple, noncrossing arcs $\beta_i$ so that each arc $\beta_i$ is either a corner arc of the triangle, or is an arc homotopic to an interval in the boundary of the triangle.

As shown in \cite[Section 12]{Hig23}, for each such $\beta_{i}$ we have $F_{\mathfrak{T}}(\beta_{i})=\beta_{i}^{(N)}.$ Since $F_\mathfrak{T}$ is an algebra homomorphism, we obtain that $F_{\mathfrak{T}}(T)= \prod \beta_{i}^{(N)}.$ Then as shown in \cite[Section 7.1]{Wang23}, the framed arcs $\beta_{i}^{(N)}$ satisfy the same special moves shown in (\ref{q=1 moves}). So, by applying the moves used to transform each $\alpha_i$ into $\beta_i$ in reverse order to the framed power $\beta_i^{(N)}$, we have $\prod \beta_{i}^{(N)}=\cup \alpha_i^{(N)} \in \Sq(\mathfrak{T}),$ as required.
\end{proof}

\subsection{Embedded bigons in stated skein modules}\label{Functoriality stated skein module}
We briefly describe a special case of the functoriality of stated skein modules described in \cite{BL22} for the case of embedded thickened bigons. Consider the bigon $\mathfrak{B}=[0,1] \times I.$ Suppose that $i: \mathfrak{B} \times I \rightarrow \Sigma \times I$ is a proper orientation preserving embedding of the thickened bigon so that $i(\partial \mathfrak{B} \times I) \subset \partial \Sigma \times I.$ Let $p: \Sigma \times I \rightarrow I$ be the height function of the thickened surface. We say that $i$ is boundary-compatible if

\begin{enumerate}
	\item it is height preserving on each boundary component: if for $x=0,1,$ we have that if $s,t \in I$ with $s<t$ then $p(i((x,y),s))<p(i((x,y),t))$ for every $y \in I,$ and
	\item when two boundary components of $\mathfrak{B}$ are sent to the same boundary component of $\Sigma,$ then each boundary component is sent to different heights: if $a,b$ are the boundary arcs of $\mathfrak{B}$ and $c$ is a boundary arc of $\Sigma$ such that $i(a \times I), i(b \times I) \subset c \times I,$ then we have that $p(i(a \times I)) \cap p(i(b \times I)) = \emptyset.$
\end{enumerate}

Such a boundary-compatible embedding $i$ induces a homomorphism of stated skein modules

\begin{equation*}
i_*: \Sq(\mathfrak{B} \times I) \rightarrow \Sq(\Sigma \times I).
\end{equation*}

The stated skein module homomorphism $i_*$ need not be an algebra homomorphism and it need not be injective. The construction extends easily to the case that $i: \sqcup (\mathfrak{B} \times I) \rightarrow \Sigma \times I$ is a boundary-compatible embedding of multiple thickened bigons.

\subsection{The Frobenius image of a knot diagram on a punctured surface}

Suppose $\Sigma$ is a surface so that each component has at least one puncture and also assume that no connected component of $\Sigma$ is the 1-punctured sphere or is the 2-punctured sphere. Such a surface admits an ideal triangulation. We choose an arbitrary but fixed ideal triangulation of $\Sigma$ consisting of edge set $\mathcal{E}$ of ideal arcs. So $\Sigma$ is obtained from gluing a finite disjoint union of triangles $\mathfrak{T}_i$ along edges. Let $g: \sqcup (\mathfrak{T}_i \times I) \rightarrow \Sigma \times I$ be the gluing map associated to the triangulation.

Let $K \subset \Sigma \times (-1,1)$ be a fixed representative of an isotopy class of a framed knot so that the representative satisfies $K \cap (\mathcal{E} \times I) \neq \emptyset$ and such that its points of intersection with the triangulation are transverse and vertically framed, and such that the points of intersection on the same component of $\mathcal{E} \times I$ occur at distinct heights. Thus, $K$ is in a good position for applying the splitting homomorphism to the edges $\mathcal{E}$ of the triangulation, and we also have that $K$ is cut in at least one spot when applying the splitting homomorphism.

By definition, the framed knot $K$ is supported by an annulus embedded in $\Sigma \times I$ with horizontal framing. The embedding of the annulus can be extended to an embedding of the thickened annulus $k: \mathcal{A}\times I \hookrightarrow \Sigma \times I$ such that $k \vert _{\mathcal{A} \times \{0\}}$ gives the support of $K$ and so that for each ideal arc $e \in \mathcal{E},$ $k^{-1}(e \times I)=a_e \times I \subset \mathcal{A} \times I$ for an ideal arc $a_e \subset \mathcal{A}$. We further arrange for $k$ to be height preserving on $k^{-1}( \mathcal{E} \times I).$

Consider the gluing map $g: \sqcup (\mathfrak{T}_i \times I) \rightarrow \Sigma \times I$ associated to the triangulation and consider the preimage $g^{-1}(k(\mathcal{A} \times I)).$ Since $k$ is cut by the triangulation in at least one spot, we see that the connected components of the preimage $g^{-1}(im(k))$ are thickened bigons and $g^{-1}(im(k))=\sqcup (\mathfrak{B}_j \times I).$ Thus, our construction of $k$ induces a boundary-compatible embedding $\bar{k}:\sqcup (\mathfrak{B}_j \times I) \hookrightarrow \sqcup (\mathfrak{T}_i \times I$) so that the following square commutes.

\begin{equation}\label{commutative embedding}
\begin{tikzcd}
\sqcup (\mathfrak{T}_i \times I) \arrow[r, "g"] & \Sigma \times I\\
\sqcup (\mathfrak{B}_j \times I) \arrow[r, "\bar{g}"] \arrow[u, hookrightarrow, "\bar{k}"] & \mathcal{A} \times I \arrow[u, hookrightarrow, "k"]
\end{tikzcd}
\end{equation}

See the following diagram for an illustration of the commutative square.

\resizebox{\textwidth}{!}{
\begingroup%
  \makeatletter%
  \providecommand\color[2][]{%
    \errmessage{(Inkscape) Color is used for the text in Inkscape, but the package 'color.sty' is not loaded}%
    \renewcommand\color[2][]{}%
  }%
  \providecommand\transparent[1]{%
    \errmessage{(Inkscape) Transparency is used (non-zero) for the text in Inkscape, but the package 'transparent.sty' is not loaded}%
    \renewcommand\transparent[1]{}%
  }%
  \providecommand\rotatebox[2]{#2}%
  \newcommand*\fsize{\dimexpr\f@size pt\relax}%
  \newcommand*\lineheight[1]{\fontsize{\fsize}{#1\fsize}\selectfont}%
  \ifx\svgwidth\undefined%
    \setlength{\unitlength}{378.75660333bp}%
    \ifx\svgscale\undefined%
      \relax%
    \else%
      \setlength{\unitlength}{\unitlength * \real{\svgscale}}%
    \fi%
  \else%
    \setlength{\unitlength}{\svgwidth}%
  \fi%
  \global\let\svgwidth\undefined%
  \global\let\svgscale\undefined%
  \makeatother%
  \begin{picture}(1,0.73434962)%
    \lineheight{1}%
    \setlength\tabcolsep{0pt}%
    \put(0,0){\includegraphics[width=\unitlength,page=1]{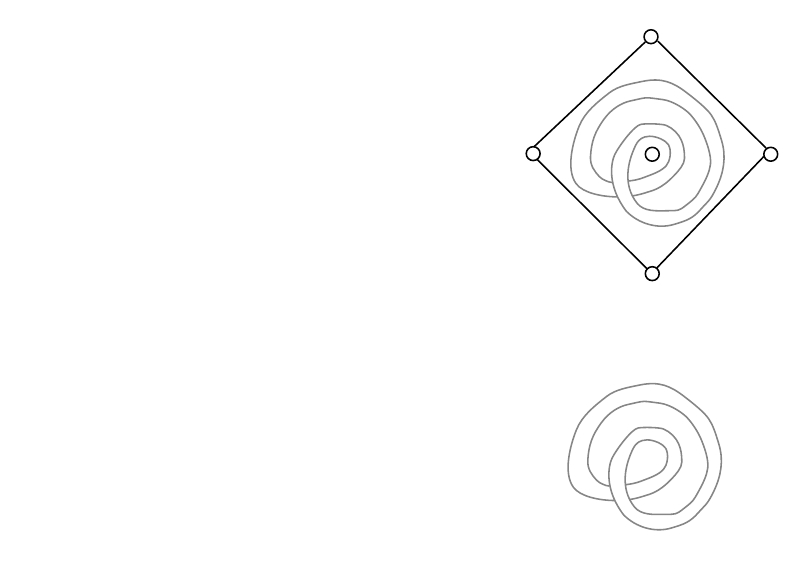}}%
    \put(0.78103341,0.31801165){\color[rgb]{0,0,0}\makebox(0,0)[lt]{\lineheight{1.25}\smash{\begin{tabular}[t]{l}$k \hookuparrow$\end{tabular}}}}%
    \put(0,0){\includegraphics[width=\unitlength,page=2]{embedding.pdf}}%
    \put(0.48395179,0.53486042){\color[rgb]{0,0,0}\makebox(0,0)[lt]{\lineheight{1.25}\smash{\begin{tabular}[t]{l}$\underset{\rightarrow}{g}$\end{tabular}}}}%
    \put(0.48554907,0.12595305){\color[rgb]{0,0,0}\makebox(0,0)[lt]{\lineheight{1.25}\smash{\begin{tabular}[t]{l}$\underset{\rightarrow}{\bar{g}}$\end{tabular}}}}%
    \put(0.13893625,0.32082309){\color[rgb]{0,0,0}\makebox(0,0)[lt]{\lineheight{1.25}\smash{\begin{tabular}[t]{l}$\bar{k} \hookuparrow$\end{tabular}}}}%
  \end{picture}%
\endgroup%

}

Recall that gluing maps $g$ and $\bar{g}$ induce splitting homomorphisms $g_*$ and $\bar{g}_*.$  If $\bar{g}$ involves the gluing of $l$ bigons into an annulus, then $\bar{g}_*$ is the map $\Delta^{(l)}$ from Corollary \ref{Frobenius annulus}. On the other hand, the embeddings $k$ and $\bar{k}$ induce homomorphisms of (stated) skein modules $k_*$ and $\bar{k}_*$ from Sections \ref{Functoriality skein module} and \ref{Functoriality stated skein module}, respectively.

\begin{lemma}\label{cutting and gluing annulus}
 We have that the following square of skein modules commutes.

\begin{equation*}
\begin{tikzcd}
\otimes \Sq(\mathfrak{T}_i \times I ) \arrow[r, leftarrow, "\Delta_\mathcal{E}"] \arrow[d, leftarrow, "\bar{k}_*"] & \Sq(\Sigma \times I) \\
\otimes \Sq(\mathfrak{B}_j \times I) \arrow[r, leftarrow, "\Delta^{(l)}"] & \Sq(\mathcal{A} \times I) \arrow[u,"k_*"]\\
\end{tikzcd}
\end{equation*}
\end{lemma}

\begin{proof}
Since each map is an $\mathcal{R}\text{-module}$ homomorphism, we will show the square commutes by verifying that it commutes on the spanning set of classes of web diagrams in $\Sq(\mathcal{A}).$ Let $[W] \in \Sq(\mathcal{A})$ be an element represented by a web $W$ whose edges are vertically framed. Thus, $k_*([W])$ is a web in $\Sq(\Sigma)$ and we may place it in valid position for the splitting homomorphism. By definition of $\Delta^{(l)},$ we have that $\Delta^{(l)}([W])=\sum [(W_B,s)]$ is a state sum for a lift $W_B$ of $W$ satisfying $\bar{g}(W_B)=W.$ Then by definition of $\bar{k}_*,$ we have that
\begin{equation*}
\bar{k}_*\circ \Delta^{(l)}([W])=\sum_{\text{admissible } s} [(\bar{k}(W_B),s)] 	
\end{equation*}
On the other hand, let $W_T$ be a lift of $k(W)$ such that $g(W_T)=k(W).$ Then we have that
\begin{equation*}
\Delta_{\mathcal{E}} \circ  k_*([W])=\sum_{\text{admissible } t} [(W_T, t)]
\end{equation*} is a state sum over the lift. To finish our proof we must show that $\sum [(\bar{k}(W_B),s)]  =\sum [(W_T, t)].$ To this end, consider
\begin{equation*}
(g\circ\bar{k})(W_B)=(k\circ \bar{g})(W_B)=k(W),
\end{equation*} where the first equality holds by commutativity of the square (\ref{commutative embedding}). Thus, we have observed that $\bar{k}(W_B)$ and $W_T$ are both lifts of $k(W).$ By the well-definedness of the splitting homomorphism, we have that the state sums are equal as claimed, $\sum [\bar{k}(W_B),s]  =\sum [W_T, t].$
\end{proof}

Next, we assemble the ingredients we have collected so far and we consider the following diagram.

\begin{proposition}\label{commutative cube}
The following cube is commutative.

\begin{equation*}
	\begin{tikzcd}
		\displaystyle\bigotimes_{i=1}^l\Sc(\mathfrak{B}) \arrow[rrr, "\bar{k}_*"] \arrow[ddd,"F_\mathfrak{B}^{\otimes l}"] \arrow[dr,leftarrow, "\Delta^{(l)}"] & & & \displaystyle\bigotimes_{i}\Sc(\mathfrak{T}) \arrow[dl,leftarrow, "\Delta_{\mathcal{E}}"] \arrow [ddd, "\otimes F_{\mathfrak{T}}"]\\
		& \Sc(\mathcal{A}) \arrow[r, "k_*"] \arrow[d, "F_\mathcal{A}"] & \Sc(\Sigma) \arrow[d, "F_\Sigma"] & \\
		& \Sq(\mathcal{A}) \arrow[r, "k_*"] \arrow[dl, "\Delta^{(l)}"] & \Sq(\Sigma) \arrow[dr, "\Delta_{\mathcal{E}}"] & \\
		\displaystyle\bigotimes_{i=1}^l\Sq(\mathfrak{B}) \arrow[rrr, "\bar{k}_*"] & & & \displaystyle\bigotimes_i \Sq(\mathfrak{T}) 
	\end{tikzcd}
\end{equation*}
\end{proposition}

\begin{proof}
There are $6$ faces to check, which we will call the top, bottom, left, right, outer, and center. The top and bottom faces commute by Lemma \ref{cutting and gluing annulus}, which holds for any specialization of $q.$ The right face commutes by Definition \ref{Frobenius definition} of $F_\Sigma$. The left face commutes due to Corollary \ref{Frobenius annulus}. The outer face commutes due to Proposition \ref{Frobenius triangle} along with the definition of $F_{\mathfrak{B}}$ from Section \ref{Frobenius bigon} paired with the definition of $\bar{k}.$

The last thing we need to show is the commutativity of the center face. We must show that $k_* \circ F_{\mathcal {A}}= F_\Sigma \circ k_*.$ By the injectivity of $\Delta_{\mathcal{E}},$ it suffices to show that $\Delta_{\mathcal{E}} \circ k_* \circ F_{\mathcal {A}}= \Delta_{\mathcal{E}} \circ F_\Sigma \circ k_*.$ We will do this by using the commutativity of the other squares in the following way

\begin{align*}
\Delta_{\mathcal{E}} \circ k_* \circ F_{\mathcal {A}}&=\bar{k}_* \circ \Delta^{(l)} \circ F_{\mathcal{A}} & \text{(bottom face)}\\
&=\bar{k}_* \circ F_{\mathfrak{B}}^{\otimes l} \circ \Delta^{(l)} & \text{(left face)}\\
&=\otimes F_{\mathfrak{T}} \circ \bar{k}_* \circ \Delta^{(l)} & \text{(outer face)}\\
&=\otimes F_{\mathfrak{T}} \circ \Delta_{\mathcal{E}} \circ k_* & \text{(top face)}\\
&=\Delta_{\mathcal{E}} \circ F_{\Sigma} \circ k_* & \text{(right face)},
\end{align*}
as required.
\end{proof}

The previous proposition gives us the tool to prove the following.

\begin{theorem}\label{Frobenius knot}
Suppose that $\Sigma$ is an ideal triangulable surface and that $q^{1/3}$ is a root of unity of order coprime to $6,$ then the image of $F_\Sigma$ on an element represented by an oriented framed knot $K \subset \Sigma \times I$ is the element $K^{[P^{(N)}]}$ obtained by the threading the power sum polynomial $P^{(N)}$ along $K.$
\end{theorem}

\begin{proof}
Pick an ideal triangulation $\mathcal{E}$ of $\Sigma$ and isotope $K$ to give it the vertical framing and so that $K \cap (\mathcal{E} \times I) \neq \emptyset$ and so it intersects the triangulation transversely. Let $k: \mathcal{A} \times I \hookrightarrow \Sigma \times I$ be an embedding of a thickened annulus chosen so that the image of $k$ is horizontal and so that $k$ sends a vertically framed core loop of the annulus to $K$: $k(l_+)=K.$ We can then apply Proposition \ref{commutative cube}, focusing on the center square of the diagram. We have that in $\Sc(\Sigma)$ we have $[K]=k_*([l_+]).$ We see that
\begin{align*}
F_\Sigma(K)&=F_\Sigma(k_*(l_+))\\
&=k_*F_\mathcal{A}(l_+)\\
&=k_*P^{(N)}(l_+,l_-)\\
&=K^{[P^{(N)}]},  	
\end{align*} where the second equality comes from the commutativity of the center square in Proposition \ref{commutative cube}, the third inequality comes from Corollary \ref{Frobenius annulus}, the final inequality comes from the fact that powers of the core loop in the skein algebra of the annulus correspond to taking copies in the direction of the framing.
\end{proof}

\subsection{The Frobenius image of a link}

We now extend the previous result to the case of multiple knot components.

\begin{corollary}\label{Frobenius link}
Assume the same hypotheses as Theorem \ref{Frobenius knot}. Then if $L \subset \Sigma \times I$ is a framed link, then the image $F_{\Sigma}(L)$ is obtained by threading the power sum polynomial $P^{(N)}$ along each component of $L.$
\end{corollary}

\begin{proof}
Consider the link $L \subset \Sigma \times I$ with $L=\sqcup_{i=1}^l K_i$ for disjoint oriented framed knots $K_i \subset \Sigma \times I.$ When $q=1,$ the skein theory allows crossing changes and $L$ can be realized as a product of knots
\begin{equation*}
L=\prod_{i=1}^l K_i' \in \Sc(\Sigma)
\end{equation*} by isotoping each $K_i$, passing through any other knot component to arrive at $K_i' \subset \Sigma \times (\frac{i-1}{l},\frac{i}{l}).$  Since $F_\Sigma$ is an algebra homomorphism, we have that $F_\Sigma(L)= \prod_{i=1}^l F_\Sigma(K_i').$ By Theorem \ref{Frobenius knot}, $F_\Sigma(K_i')=K_i'^{[P^{(N)}]}.$ By the result of \cite[Theorem 1]{BH23}, each $K_i'^{[P^{(N)}]}$ is transparent in $\Sq(\Sigma)$ and so the elements can be passed through each other to obtain $\prod_{i=1}^l K_i'^{[P^{(N)}]}=L^{[P^{(N)}]} \in \Sq(\Sigma).$ We then conclude that $F_\Sigma(L)= L^{[P^{(N)}]}$ as claimed.
\end{proof}

\subsection{The Frobenius homomorphism for closed surfaces}

\begin{theorem}\label{Frobenius surface}
Suppose that $q^{1/3}$ is a root of unity of order $N$ coprime to $6.$ For any surface $\Sigma$, including a closed surface, there is an algebra homomorphism $F_{\Sigma}: \Sc(\Sigma) \rightarrow \Sq(\Sigma)$ such that the image of any link $L \subset \Sigma \times I$ is given by the threading of the power sum polynomial $P^{(N)}$ along each component of $L,$ \begin{equation*}
F_\Sigma(L)=L^{[P^{(N)}]}.
\end{equation*}
Furthermore, the image of $F_\Sigma$ is contained in the center of $\Sq(\Sigma).$
\end{theorem}

\begin{proof}

We will prove this by showing that the threading map respects the process of filling in a puncture. We will show that if the Theorem \ref{Frobenius surface} holds in the case of $\Sigma_{g,p+1}$, for a genus $g$ surface with $p+1$ punctures, then it holds in the case of $\Sigma_{g,p}$ as well. This will suffice since for any $g\geq0$ there exists $p\geq 0$ such that $\Sigma_{g,p+1}$ is ideal triangulable and thus satisfies Theorem \ref{Frobenius link}.

Let $x$ be a distinguished puncture of $\Sigma_{g,p+1}$ so that the remaining $p$ punctures are the same as the punctures of $\Sigma_{g,p}.$ We have that $\Sq(\Sigma_{g,p})$ is the quotient of $\Sq(\Sigma_{g,p+1})$ by the ideal generated by the additional relation that allows one to sweep a strand (of any orientation) over the distinguished puncture $x$ for $\Sigma_{g,p+1}$:

\begin{equation*}
\begin{tikzpicture}[baseline=3ex]
	\draw (0,1) to [out =-45, in =90] (1/4,1/2) to [out=-90=, in=45] (0,0);
	\filldraw [fill=white, draw=black] (0,1/2) circle [radius=.05];
\end{tikzpicture}=
\begin{tikzpicture}[baseline=3ex]
\draw (0,1) to [out=-135, in =90] (-1/4,1/2) to [out=-90, in=135] (0,0) ;
\filldraw [fill=white, draw=black] (0,1/2) circle [radius=.05];
\end{tikzpicture}
\end{equation*}

By our assumption, we have an algebra homomorphism $F_{\Sigma_{g,p+1}}: \Sc(\Sigma_{g,p+1}) \rightarrow \Sq(\Sigma_{g,p+1})$ defined by threading $P^{(N)}$ along components of links. In order to show  that this homomorphism induces a homomorphism on $F_{\Sigma_{g,p}}$ we need to show that the map respects the sweep move.

Suppose that $W$ and $W'$ are two webs in $\Sc(\Sigma_{g,p+1})$ which differ only by a sweep move. We can obtain the image of $F_{\Sigma_{g,p+1}}$ on these webs by rewriting $W$ and $W'$ as links in any way we choose, since we have assumed $F_{\Sigma_{g,p+1}}$ is well-defined. Since $W$ and $W'$ are identical outside of a neighborhood of the puncture $x,$ we choose to rewrite $W=\sum b_i L_i$ and $W'=\sum b_i L_i'$ where $L_i$ and $L_i'$ are links which differ only by a sweep move. This is possible since rewriting a web as a linear combination of links can be done by repeatedly applying the move (\ref{2gon}) in the form

\begin{equation*}
\begin{tikzpicture}[baseline=3ex]
	\draw [<-] (0,0)--(1/4,1/4);
	\draw [<-] (1/2,0)--(1/4,1/4);
	\draw (1/4,1/4)--(1/4,3/4);
	\draw [-<] (1/4,3/4)--(1/4,1/2);
	\draw (1/4,3/4)--(0,1);
	\draw (1/4,3/4)--(1/2,1);
	\draw [->] (0,1)--(1/8,7/8);
	\draw [->] (1/2,1)--(3/8,7/8);
\end{tikzpicture}=
\begin{tikzpicture}[baseline=3ex]
\node (center) at (1/4,1/2) {};
\draw[->] (1/2,1) -- (0,0);
\draw (0,1) -- (center);
\draw [->] (center)--(1/2,0);
\end{tikzpicture}
-\begin{tikzpicture}[baseline=3ex]
\draw [->] (0,1)--(0,0);
\draw [->] (1/2,1)--(1/2,0);
\end{tikzpicture}, 
\end{equation*}

along edges of the web, which can be chosen while keeping fixed the edge involved in the sweep move.

Then $F_{\Sigma_{g,p+1}}(W)=\sum b_{i} L_i^{[P^{(N)}]}$ and $F_{\Sigma_{g,p+1}}(W')=\sum b_{i} L_i'^{[P^{(N)}]}.$ Since these two elements are related to each other by repeated applications of sweep moves, we see that $F_{\Sigma_{g,p+1}}$ descends to an algebra map $F_{\Sigma_{g,p}}: \Sc(\Sigma_{g,p}) \rightarrow \Sq(\Sigma_{g,p})$ which is defined as claimed.

Since $\Sq(\Sigma_{g,p})$ is an algebra quotient of $\Sq(\Sigma_{g,p+1})$ and $F_{\Sigma_{g,p+1}}$ has central image, we obtain that $F_{\Sigma_{g,p}}$ has central image as well. Alternatively, the fact that the threading operation produces transparent elements was shown in \cite[Theorem 1]{BH23}.

\end{proof}

\subsection{The Frobenius homomorphism for skein modules of 3-manifolds}

\begin{theorem}\label{Frobenius manifold}
Suppose that $q^{1/3}$ is a root of unity of order $N$ coprime to $6.$ For any oriented 3-manifold $M$, there is a homomorphism of skein modules $F_{M}: \Sc(M) \rightarrow \Sq(M)$ such that that the image of any oriented framed link $L \subset M$ is given by the threading of the power sum polynomial $P^{(N)}$ along each component of $L,$ \begin{equation*}
	F_M(L)=L^{[P^{(N)}]}.
\end{equation*}

Furthermore, the image of $F_M$ is contained in the submodule of transparent elements of $\Sq(M).$
\end{theorem}

\begin{proof}

Let $W \subset M$ be a web representing an element of the skein module $\Sc(M).$ By definition, the web $W$ is equipped with a thickened oriented surface $\Sigma_{W}\subset M$ containing $W$ and deformation retracting onto it. Since $\Sigma_W$ is a thickened surface, by Theorem \ref{Frobenius surface} we have a well-defined Frobenius homomorphism $F_{\Sigma_W}.$ Let $w: \Sigma_{W} \hookrightarrow M$ be the natural inclusion of 3-manifolds. We will define $F_{M}(W)$ to be the element $w_*(F_{\Sigma_{W}}(W)) \in \Sq(M).$ So $F_M(W)$ can be computed by rewriting $W$ in $\Sc(\Sigma_{W})$ as any linear combination of links in a small neighborhood of $W$, and then applying the operation of threading $P^{(N)}$ along the links. To finish the proof, we must show this construction of $F_M$ is well-defined on $\Sc(M).$ 

It is clear that $F_M$ respects any isotopy of the web $W$ in $M.$ We need to show that the definition of $F_M$ respects the defining skein relations of $\Sc(M).$ Consider a skein relation of (\ref{2gon})-(\ref{loop}) represented by a linear combination of webs $\sum a_iW_i$ satisfying
\begin{equation*}
\sum a_i[W_i]=0 \in \Sc(M),
\end{equation*}  such that the webs $W_i$ are identical to each other outside of a thickened disk in $M.$

To show that $F_M$ respects the skein relation, our goal is to show that

\begin{equation}\label{goal}
\sum a_i {w_i}_*(F_{\Sigma_{W_i}}([W_i]))=0 \in \Sq(M).
\end{equation}

Each web $W_i$ is equipped with a thickened surface $\Sigma_{W_i} \subset M$ identical to each other outside of a thickened disk  $D \times I \subset M.$ For any fixed $k,$ we can glue $D \times I$ to $\Sigma_{W_k} \setminus (\Sigma_{W_k} \cap (D \times I))$ to obtain a thickened surface $\Sigma_{\hat{W}}$ which contains each of the thickened surfaces $\Sigma_{W_i}.$ Let $\hat{w}: \Sigma_{\hat{W}} \hookrightarrow M$ be the inclusion and, for each $i,$ let $\hat{w}_i$ denote the inclusion $\hat{w}_i: \Sigma_{W_i} \hookrightarrow \Sigma_{\hat{W}}.$ We have that $w_i=\hat{w} \circ \hat{w}_i$ and by the functoriality of skein modules we consequently have that
\begin{equation*}
{w_i}_*=\hat{w}_* \circ \hat{w}_{i*}.
\end{equation*}

Now inside of $\Sigma_{\hat{W}},$ the linear combination of webs $\sum a_iW_i$ satisfies \begin{equation*}
\sum a_i \hat{w}_{i*} ([W_i])=0 \in \Sc(\Sigma_{\hat{W}}).
\end{equation*} Since $F_{\Sigma_{\hat{W}}}$ is well-defined, we can compute its image on $\sum a_iW_i$ by rewriting each $W_i=\sum b_{ij}L_{ij}$ as a linear combination of links in $\Sc(\Sigma_{\hat{W}})$ and then applying the threading operation. We can choose the links so that each $L_{ij}$ is contained in a neighborhood of $W_i.$ By the fact that  $F_{\Sigma_{\hat{W}}}$ is well-defined, we have that
\begin{equation*}
\sum a_i \hat{w}_{i*}(b_{ij}[L_{ij}^{[P^{(N)}]}])=0  \in \Sq(\Sigma_{\hat{W}}),
\end{equation*}
and whence

\begin{equation*}
\hat{w}_*(\sum a_i\hat{w}_{i*}(b_{ij}[L_{ij}^{[P^{(N)}]}]))=0 \in \Sq(M)
\end{equation*}
as well. Since $w_{i*}=\hat{w}_* \circ \hat{w}_{i*},$ we have shown that (\ref{goal}) holds and we see that $F_M$ respects the skein relations, as required.

To see that the image of $F_M$ is transparent, one could use the centrality of Theorem \ref{Frobenius surface}, or alternatively this was proven in \cite[Theorem 2]{BH23}.

\end{proof}

\bibliographystyle{amsalpha}
\bibliography{SL3.bib}

\end{document}